\documentclass[12pt]{article}

\usepackage{amssymb}
\usepackage{amsthm}
\usepackage{amsmath}
\usepackage{graphicx}
\usepackage{fullpage}
\usepackage{color}
 \numberwithin{equation}{section}
 \usepackage{verbatim}
\usepackage[alphabetic]{amsrefs}
\usepackage{hyperref}
\theoremstyle{plain}
\newtheorem{thm}{Theorem}[section]

\newtheorem{lem}[thm]{Lemma}
\newtheorem{prop}[thm]{Proposition}

\theoremstyle{definition}
\newtheorem{defn}[thm]{Definition}
\newtheorem{ex}[thm]{Example}

\theoremstyle{remark}

\newtheorem{rem}[thm]{Remark}

\newcommand{\R}{\mathbb{R}}

\newcommand{\FF}{\mathcal{F}}

\newcommand{\pv}{\text{P.V.}}
\newcommand{\lap}{\Delta}
\newcommand{\e}{\varepsilon}

\newcommand{\bp}{\begin{proof}[\ensuremath{\mathbf{Proof}}]}
\newcommand{\ep}{\end{proof}}

\begin{document}

\title{Large time behavior for a nonlocal nonlinear gradient flow}

\author{Feng Li and Erik Lindgren} 

\maketitle

\begin{abstract} \noindent 	We study the large time behavior of the nonlinear and nonlocal equation
$$
v_t+(-\lap_p)^sv=f \, ,
$$
where $p\in (1,2)\cup (2,\infty)$, $s\in (0,1)$ and
$$
(-\lap_p)^s v\, (x,t)=2 \,\pv \int_{\R^n}\frac{|v(x,t)-v(x+y,t)|^{p-2}(v(x,t)-v(x+y,t))}{|y|^{n+sp}}\, dy.
$$
This equation arises as a gradient flow in fractional Sobolev spaces. We obtain sharp decay estimates as $t\to\infty$. The proofs are based on an iteration method in the spirit of J. Moser previously used by P. Juutinen and P. Lindqvist.
\end{abstract}

\tableofcontents

\section{Introduction}
We study the large time behavior of solutions of the equation
\begin{equation}\label{MainPDE}
v_t+(-\Delta_p)^s v=f
\end{equation}
where $p\in (1,2)\cup (2,\infty)$, $s\in (0,1)$ and $(-\Delta_p)^s$ is the fractional $p$-Laplacian
\begin{equation}\label{pv}
(-\Delta_p)^s u\, (x):=2\, \pv \int_{\R^n}\frac{|u(x)-u(x+y)|^{p-2}(u(x)-u(x+y))}{|y|^{n+sp}}\, dy.
\end{equation}
Here $\pv$ denotes the \emph{principal value}. The operator $(-\lap_p)^s $ arises as the first variation of the Sobolev-Slobodecki\u{\i} seminorm (see Section \ref{sec:not})
\[
u\mapsto \iint_{\mathbb{R}^n\times \mathbb{R}^n} \frac{|u(x)-u(x)|^p}{|x-y|^{n+s\,p}}\,dx\,dy,
\]
and is a nonlocal (or fractional) version of the $p-$Laplace operator, 
\[
-\Delta_pu=-\mathrm{div\,}(|\nabla u|^{p-2}\nabla u).
\]
In particular, solutions of $(-\Delta_p)^s u =0$ converge to solutions of $-\Delta_p u =0$, as $s$ goes to $1$, if suitably rescaled. For details, see Section 1.4 in \cite{BL15} and \cite{IN10}.

\section{Main result}
Our main result concerns the decay rate of the difference of two solutions for $p\neq 2$. In particular, it implies the sharp convergence rate to the corresponding stationary solution.  In the case $p=2$, the exponential convergence to a stationary solution is well known. The results are different in nature if $p>2$ or $p<2$ and are therefore presented in two different theorems below. For the precise definition of weak solutions and revelant function spaces, see Section \ref{sec:weak}.

\begin{thm}\label{thm:main1}
Let $p>2$, $\Omega$ be a bounded domain and $f\in L^{p^\prime}_{\text{loc}}((0,\infty);(X^{s,p}_0(\Omega,\Omega^\prime))^*)$. Assume that $u$ and $v$ are weak solutions of 
$$
u_t+(-\lap_p)^s u = f\quad \text{ in } \Omega\times (0,\infty),$$
and that $u-v\in L^p_{\text{loc}}((0,\infty);W_0^{s,p}(\Omega))$. Then there is a constant $C=C(\Omega,n,p,s)$ such that 
$$
\sup_{t\geq T}\|v(\cdot,t)-u(\cdot,t)\|_{L^\infty(\Omega)}\leq CT^\frac{-1}{p-2}.
$$
\end{thm}

\begin{thm}\label{thm:main2}
Let $\max\{1,\frac{2n}{n+2s}\}<p<2$, $\Omega$ be a bounded domain and $f\in L^{p^\prime}_{\text{loc}}((0,\infty);(X^{s,p}_0(\Omega,\Omega^\prime))^*)$. Assume that $u$ and $v$ are weak solutions of 
$$
u_t+(-\lap_p)^s u = f\quad \text{ in } \Omega\times (0,\infty),$$
and that $u-v\in L^p_{\text{loc}}((0,\infty);W_0^{s,p}(\Omega))$ and $u,v\in  L^\infty_{\text{loc}}((0,\infty);W^{s,p}(\R^n))$. Then there are constants $\lambda(\Omega,n,p,s)$ and $C(\Omega,n,p,s)$, such that for all $t>2\tau>0$ there holds
$$
\|v( \cdot,t)-u( \cdot,t)\|_{L^\infty(\Omega)}\leq \frac{C}{(tL)^{\frac{1}{2}+\frac{p\nu}{4(p-\nu)}}}\|v(\cdot,\tau)-u( \cdot,\tau)\|_{L^2(\Omega)}e^{-\lambda L(\frac{t}{2}-\tau)},
$$
where $L=\left(\sup\limits_{\ell\geq\tau}\left([v(\ell, \cdot)]_{W^{s,p}(\mathbb{R}^n)}+[u(\ell,\cdot)]_{W^{s,p}(\mathbb{R}^n)}\right)\right)^{p-2}$ and
\begin{equation}\nonumber
\nu=\left\{\begin{matrix}
\frac{2n}{n+2s}&\quad \text{ if}\ n>sp;\\
\frac{2N}{N+2s}&\quad \text{for any $N>sp$ if}\ sp\geq n.
\end{matrix}
\right.
\end{equation}
\end{thm}
The results are sharp in the sense that the rate of convergence in Theorem \ref{thm:main1} cannot be improved and that one cannot have better than exponential convergence in Theorem \ref{thm:main2}. This is discussed in more detail together with examples in Section \ref{sec:ex}.

\begin{rem} The dependence of $\Omega$ in the constants in the theorems above is really a dependence of $|\Omega|$. This dependence enters at several stages and comes from applying embeddings from $W^{s,p}(\R^n)$ to $L^q(\Omega)$, whenever $q$ is not equal to the Sobolev exponent. See Proposition \ref{prop:sobpoin1}.
\end{rem}

\begin{rem} The constant $L$ in Theorem \ref{thm:main2} clearly depends on the solutions $u$ and $v$. It is noteworthy that unless both $u$ and $v$ are equal to the same constant function, $L$ is positive and finite. Moreover, if $u$ is a stationary solution then one may prove that the solution $v$ has a non-increasing $W^{s,p}$-seminorm. This can be done in the same spirit as it is done for the so called fractional Trudinger's equation. See Remark 6.3 in \cite{HL21}. In addition, if $u$ and $v$ are both $s$-H\"older continuous then we may write the estimate with $L$ expressed in terms of the $s$-H\"older seminorms instead.
\end{rem}

\section{Known result}
Evolutionary equations involving the fractional $p$-Laplacian has recently attracted much attention.  Equation \eqref{MainPDE} has been studied in \cite{abd,BB20, BLS21, dzz, GT21, MRT, puh, str19, str19b, Vaz15, Vaz3, Vaz20, warma}. In particular, the large time behavior of solutions to \eqref{MainPDE} with $f=0$ and with zero Dirichlet data has been studied in \cite{Vaz15} and \cite{abd}. In \cite{Vaz15}, it is proved that for $p>2$, there are solutions of the form
$$
U(x,t)=t^\frac{-1}{p-2}F(x), 
$$
where $F$ solves a stationary equation. Using this, decay estimates of the form
$$
|u(x,t)|\leq Ct^\frac{-1}{p-2},
$$
 are obtained.

In \cite{abd}, it is proved that for $1<p<2$, solutions enjoy the finite time extinction property. This is also mentioned in Section 7 in \cite{Vaz15}.

Note that all the above results are for solutions with zero Dirichlet boundary data, while our results detail the large time behavior of the difference of two solutions, and in particular also for solutions with non-zero Dirichlet data.

In passing, we also mention \cite{HLnonlocal}, where the large time behavior of solutions of the doubly nonlinear equation
$$
|\partial_t u|^{p-2}\,\partial_t u+(-\Delta_p)^s u=0, 
$$
is proved to be exponential.

We also seize the opportunity to mention the recent papers \cite{APJ1} and \cite{APJ2}, where another nonlocal doubly nonlinear equation is studied.

The corresponding local problem, that is, \eqref{MainPDE} with the $p$-Laplacian instead of the fractional $p$-Laplacian, is well understood. The large time behavior of solutions with zero Dirichlet boundary data is detailed in Chapters VI-VII in \cite{DiB}. The large time behavior for solutions with general Dirichlet data has been studied in \cite{JL09}. In particular it is proved that for $p>2$ solutions converge towards a stationary solutions with the speed $t^\frac{-1}{p-2}$, while for $p<2$ the rate of convergence is exponential. Examples also show that these rates are sharp. In the present paper we provide a non-local counterpart to their results.

\section{Preliminaries}\label{sec:not}
\subsection{Notation}
Throughout the paper we will use the notation 
$$
J_p(t)=|t|^{p-2}t, \quad p\in (1,\infty).
$$
Moreover, for $x\in \R$ and $M>0$ we define
$$
x_M = \begin{cases} x, & -M<x<M,\\
M, & x\geq M,\\
-M, &x\leq M.
\end{cases}
$$
For convenience, we also introduce 
$$
J_p^M(t)=J_p(t_M).
$$
In order to avoid cumbersome notation and long formulas we will often use $d\mu$ for short for $|x-y|^{-n-ps} dx dy$. 
\subsection{Fractional Sobolev spaces}
For $1<p<\infty$ and $s\in (0,1)$ the norm in the fractional Sobolev spaces $W^{s,p}(\R^n)$ is given by
$$
\|u\|^p_{W^{s,p}(\R^n)}=[u]^p_{W^{s,p}(\R^n)}+\|u\|^p_{L^p(\R^n)},
$$
where the Sobolev-Slobodecki\u{\i} seminorm is defined by
$$
[u]^p_{W^{s,p}(\R^n)}=\int_{\R^n}\!\!\int_{\R^n} \frac{|u(y)-u(x)|^{p}}{|y-x|^{ sp+n}} d x d y=\int_{\R^n}\!\!\int_{\R^n} |u(y)-u(x)|^{p} d\mu.
$$
The space $W^{s,p}_0(\Omega)$ is the closure of $C_0^\infty(\Omega)$ with respect to the norm $\|\cdot \|_{W^{s,p}(\R^n)}$.

\subsection{Sobolev inequalities}

We first state the fractional Sobolev inequality, this result can be found in Theorem 1 in \cite{Maz}.

\begin{thm}\label{thm:sob} Assume $sp<n$ and take $f \in W_0^{s,p}(\R^n)$. Then 
$$
\|f\|^p_{L^{p^*}(\R^n)}\leq \gamma(n,p)\frac{s(1-s)}{(n-sp)^{p-1}}[f]_{W^{s,p}(\R^n)}^p,
$$
where $p^* = \frac{np}{n-sp}$.
\end{thm} 

We also include a series of Sobolev-type inequalities.  This can for instance be found in \cite{BLS}, Proposition 2.7.

\begin{prop}
\label{prop:sobpoin1}
Suppose $1<p<\infty$ and $0<s<1$. Let $\Omega\subset\mathbb{R}^n$ be an open and bounded set. For every $u\in W^{s,p}(\mathbb{R}^n)$ such that $u=0$ almost everywhere in $\mathbb{R}^n\setminus\Omega$, we have
\begin{equation}
\label{poin}
\|u\|^p_{L^p(\Omega)}\le C_1\,|\Omega|^\frac{s\,p}{n}\,[u]^p_{W^{s,p}(\mathbb{R}^n)},
\end{equation}
\[
\|u\|^p_{L^\infty(\Omega)}\le C_1\,|\Omega|^{\frac{s\,p}{n}-1}\,[u]^p_{W^{s,p}(\mathbb{R}^n)}, \qquad \mbox{ if }  s\,p>n,
\]
and
\[
\|u\|^p_{L^{q}(\Omega)}\le C_2\,|\Omega|^{\frac{p}{q}+\frac{s\,p}{n}-1}\,[u]^p_{W^{s,p}(\mathbb{R}^n)}, \qquad \mbox{ for every } 1\le q<\infty, \mbox{ if }  s\,p=n
\]
for constants $C_1=C_1(n,p,s)$, $C_2=C_2(n,p,s)$ and $C_3=C_3(n,p,s)$.
\end{prop}

\subsection{Parabolic Banach spaces}\label{pBspaces}
Let $I\subset \mathbb{R}$ be an interval and let $V$ be a separable, reflexive Banach space, with its norm $\|\cdot\|_V$. 
Suppose that $v$ is a mapping such that for a.e. $t\in I$, $v(t)$ belongs to $V$. If the function $t\mapsto \|v(t)\|_V$ is measurable on $I$ and $1\le p\le \infty$, then we say that $v$ is an element of the Banach space $L^p(I;V)$ if 
\[
\int_I\|v(t)\|_V^pdt<+\infty.
\]
We write $v\in C(I;V)$ if the mapping $t\mapsto v(t)$ is continuous with respect to the norm on $V$. 

\subsection{Parabolic Sobolev inequalities}
We will be needing the two following parabolic inequalities of Sobolev-type:
\begin{thm}\label{thm:parsob}
Let $sp<n$ and $\Omega\subset\mathbb{R}^n$ be an open and bounded set. Then for any $f\in L^\infty((a,b);L^r(\R^n))\cap L^p((a,b);W_0^{s,p}(\Omega))$ there holds
$$
\int_{a}^b \int_\Omega |f|^{\kappa p}\leq C(n,p,s)\left(\int_a^b [f]_{W^{s,p}(\R^n)}^p \right)\left(\sup_{t\in (a,b)}\int_\Omega |f|^\frac{p\kappa^*(\kappa-1)}{\kappa^*-1}\right)^\frac{\kappa^*-1}{\kappa^*},
$$
for any $\kappa\in [1,\kappa^*]$ where $\kappa^*=n/(n-sp)$ and $r=\frac{p\kappa^*(\kappa-1)}{\kappa^*-1}$.
\end{thm}

\begin{proof}
This is a direct consequence of the Sobolev inequality. Indeed, let $q=p(1+m/n)$ and apply H\"older's inequality combined with the above Sobolev inequality to obtain	 
\[
\begin{split}
\int_{a}^b \int_\Omega |f|^q dx dt &=\int_{a}^b \int_\Omega |f|^{p+mp/n} dx dt \\
&\leq \int_{a}^b\left(\int_\Omega |f|^{np/(n-sp)} dx\right)^{(n-sp)/n}\left(\int_\Omega |f|^\frac{m}{s} dx\right)^{sp/n}dt\\
&\leq C(n,p,s) \int_{a}^b [f]_{W^{s,p}(\R^n)}^p dt \left(\sup_{t\in [a,b]} \int_\Omega |f|^\frac{m}{s} dx\right)^{sp/n}.
\end{split}
\]
With $\kappa = 1+m/n$ this is exactly the desired result.
\end{proof}

\begin{thm}\label{thm:parsob2}
Let $n\leq sp<N$ and $\Omega\subset\mathbb{R}^n$ be an open and bounded set. Then for any $f\in L^\infty((a,b);L^r(\R^n))\cap L^p((a,b);W_0^{s,p}(\Omega))$ there holds
$$
\int_{a}^b \int_\Omega |f|^{\kappa p}\leq C(\Omega,n,p,s)\left(\int_a^b [f]_{W^{s,p}(\R^n)}^p \right)\left(\sup_{t\in (a,b)}\int_\Omega |f|^\frac{p\kappa^*(\kappa-1)}{\kappa^*-1}\right)^\frac{\kappa^*-1}{\kappa^*}
$$
for any $\kappa\in [1,\kappa^*]$ where $\kappa^*=N/(N-sp)$ and $r=\frac{p\kappa^*(\kappa-1)}{\kappa^*-1}$.
\end{thm}
\begin{proof} Proceed as in the proof of Theorem \ref{thm:parsob} with replaced $n$ by $N$ and then apply the second inequality in Proposition \ref{prop:sobpoin1}.
\end{proof}

\section{Weak solutions}\label{sec:weak} 
Let $\Omega\subset \mathbb{R}^n$ be a bounded open set in $\mathbb{R}^n$ and assume that $\Omega\Subset\Omega^\prime\subset\R^n$, where $\Omega^\prime$ is a bounded open set in $\R^n$. Define also $I=(t_0,t_1]$, for any $t_0,t_1\in\mathbb{R}$ with $t_0<t_1$. In order to define our class of solutions we need to introduce the {\it tail space}
\[
L^{q}_{\alpha}(\mathbb{R}^n)=\left\{u\in L^{q}_{\rm loc}(\mathbb{R}^n)\, :\, \int_{\mathbb{R}^n} \frac{|u|^q}{1+|x|^{n+\alpha}}\,dx<+\infty\right\},\qquad q\ge 1 \mbox{ and } \alpha>0,
\]
together with the norm
\[
\|u\|_{L_\alpha^{q}(\mathbb{R}^n)} = \left(\int_{\mathbb{R}^n} \frac{|u|^q}{1+|x|^{n+\alpha}}\,dx\right)^{\frac{1}{q}}.
\]
Define also
$$
X^{s,p}_{0}(\Omega,\Omega^\prime)=\big\{v\in W^{s,p}(\Omega^\prime)\cap L^{p-1}_{sp}(\R^n): v=0\  \text{ on}\ \R^n\setminus\Omega\big\}.
$$
\begin{defn}\label{locweak}  Let $f\in L^{p^\prime}(I;(X^{s,p}_0(\Omega,\Omega^\prime))^*)$. 
We say that $u$ is a \emph{weak solution} of
\begin{equation}\label{locweaksol}
\partial_t u + (-\Delta_p)^su = f,\qquad\mbox{ in }\Omega\times I,
\end{equation}
if for any closed interval $J=[T_0,T_1]\subset I$
\[
u\in L^p(J;W^{s,p}(\Omega^\prime))\cap L^{p-1}(J; L^{p-1}_{sp}(\R^n))\cap C(J;L^2(\Omega))
\]
and 
\begin{equation}
\begin{split}
\label{locweakeq}
-\int_J\int_\Omega u(x,t)\,\partial_t\phi(x,t)\,dx\,dt&+ \int_J\iint_{\mathbb{R}^n\times\mathbb{R}^n}J_p(u(x,t)-u(y,t))\,(\phi(x,t)-\phi(y,t))d\mu\,dt  \\   
& = \int_\Omega u(x,T_0)\,\phi(x,T_0)\,dx -\int_\Omega u(x,T_1)\,\phi(x,T_1)\,dx \\
&+ \int_J \langle f(\cdot,t),\phi(\cdot,t)\rangle\,dt,
\end{split}
\end{equation}
for any $\phi\in L^{p}(J;X^{s,p}_0(\Omega,\Omega^\prime))\cap C^1(J;L^2(\Omega))$.
\par
Finally, we say that $u$ is a \emph{weak subsolution} if in \eqref{locweakeq} we replace the equality sign with less than or equal to, for any \emph{non-negative} $\phi$ as above. A \emph{weak supersolution} is defined similarly.
\end{defn}

\section{The degenerate case: $p>2$}
\subsection{A first decay estimate}
We first study the decay of certain $L^p$-norms.
\begin{prop}\label{Prop:Degenerate}
Let $\Omega$ be a bounded domain and let $u$ and $v$ be weak solutions of 
$$
u_t+(-\lap_p)^s u = f\quad \text{ in } \Omega\times (0,\infty).$$
Assume in addition that $f\in L^{p^\prime}_{\text{loc}}((0,\infty);(X^{s,p}_0(\Omega,\Omega^\prime))^*)$ and $w=u-v\in L^p_{\text{loc}}((0,\infty);W_0^{s,p}(\Omega))$. Define
$$
I_m(t)=\int_\Omega |w(x,t)|^m dx, \quad m\geq 2.
$$
Then
$$
I_m(t_2)\leq \min\left(C(t_2-t_1)^\frac{-m}{p-2},I_m(t_1)\right).
$$ 
for $t_2\ge t_1 >0$ and where $C=(\Omega, m,n,p,s)$.
\end{prop}
\begin{proof}
The first step is to estimate the derivative of $I_m$. Define 
$$
I_m^M(t)=\int_\Omega |w_M(x,t)|^m dx,
$$
where $w_M$ is the $M$-truncation of $w$ (see Section \ref{sec:not}). We will obtain an $M$-independent estimate for $I_m^M(t)$, which implies the desired result. We test the equation with $J_m(w_M)$. By Lemma \ref{lem:test} we obtain for $T_1\geq T_0$
\begin{align}
-m\int^{T_1}_{T_0}\iint_{\R^{N}\times\R^{N}}&\Big(J_p(u(x,t)-u(y,t))-J_p(v(x,t)-v(y,t))\Big)\nonumber\\
&\times\Big(J^M_{m}(u(x,t)-v(x,t))-J^M_{m}(u(y,t)-v(y,t))\Big)d\mu dt\label{mdiff} \\
&=I^M_m(T_1)-I^M_m(T_0).\nonumber 
\end{align}
From Lemma \ref{lem:ineqM} 
\[
\begin{split}
(J^M_m(v(x)-u(x))-J^M_m(v(y)-u(y))\left(J_p(v(x)-v(y))-J_p(u(x)-u(y))\right)\\
\geq C(m,p)\Big||w_M(x)|^\frac{m-2}{p}w_M-|w_M(y)|^\frac{m-2}{p}w_M(y)\Big|^p.
\end{split}
\]
From here on out, we omit the $t$-dependence of $u$ and $v$. 
Hence, 
\begin{equation}\label{eq:sest}
\begin{split}
&\int^{T_1}_{T_0}\iint_{\R^{n}\times\R^{n}} (J^M_m(v(x)-u(x))-J^M_m(v(y)-u(y))\left(J_p(v(x)-v(y))-J_p(u(x)-u(y))\right) d\mu dt\\
&\geq C(m,p)\int^{T_1}_{T_0} [|w_M|^{\frac{m-2}{p}}w_M]_{W^{s,p}(\R^n)}^p dt.
\end{split}
\end{equation}
Using that $\||w_M|^{\frac{m-2}{p}}w_M\|_{L^p(\Omega)}\leq C[|w_M|^{\frac{m-2}{p}}w_M]_{W^{s,p}(\R^n)}$, with $C=C(\Omega,n,p,s)$ (cf. \eqref{poin}) together with \eqref{mdiff} and \eqref{eq:sest} and inserted into \eqref{mdiff}, we obtain 
$$
I^M_m(T_1)-I^M_m(T_0)\leq -C\int_{T_0}^{T_1}\|w_M^{\frac{m-2}{p}}w_M\|^p_{L^p}, \quad C=C(\Omega, m,n,p,s).
$$
H\"older's inequality implies
$$
\left(\int_\Omega |w_M|^m dx\right)^\frac{m+p-2}{m}\leq C\int_\Omega |w_M|^{m+p-2} dx,\quad  C=C(\Omega,m,p).
$$
Hence, 
$$
I^M_m(T_1)-I^M_m(T_0)\leq -C\int_{T_0}^{T_1}\left(\int_\Omega |w_M|^m dx\right)^\frac{m+p-2}{m}dt=-C\int_{T_0}^{T_1} (I^M_m)^\frac{m+p-2}{m} dt, \quad C=C(\Omega, m,n,p,s).
$$
In particular, for a.e. $t$
$$
(I^M_m)'(t)\leq -C(I^M_m)^\frac{m+p-2}{m}, \quad C=C(\Omega, m,n,p,s).
$$
This implies
$$
\left((I^M_m)^\frac{2-p}{m}\right)'\geq C,\quad C=C(\Omega, m,n,p,s),
$$
for a.e. $t$. Since $I_m^M$ is a non-decreasing function, we have
$$
(I^M_m(t_2))^\frac{2-p}{m}\geq (I^M_m(t_1))^\frac{2-p}{m}+C(t_2-t_1).
$$
This implies the desired estimate for $I_m^M$ and ends the proof.
\end{proof}

\subsection{The $L^\infty$-estimate via Moser iteration} We now perform a Moser-type iteration in order to obtain an $L^\infty$-estimate for the difference of two solutions.

\begin{prop} \label{prop:degmoser} Let $\Omega$ be a bounded domain and let $u$ and $v$ be weak solutions of 
$$
u_t+(-\lap_p)^s u = f\quad \text{ in } \Omega\times (0,\infty).$$
Assume in addition that $f\in L^{p^\prime}_{\text{loc}}((0,\infty);(X^{s,p}_0(\Omega,\Omega^\prime))^*)$ and $w=u-v\in L^p_{\text{loc}}((0,\infty);W_0^{s,p}(\Omega))$. If $sp<n$ then
$$
\|w\|_{L^\infty([T	,t_2]\times \Omega)}^{\frac{n(p-2)}{sp}+p}\leq C_1T^{-\frac{n+sp}{sp}}\int_{T/2}^{t_2}\int_\Omega |w|^p dx dt,\quad C_1=C_1(n,p,s).
$$
If $n\leq sp<N$ then
$$
\|w\|_{L^\infty([T	,t_2]\times \Omega)}^{\frac{N(p-2)}{sp}+p}\leq C_2T^{-\frac{N+sp}{sp}}\int_{T/2}^{t_2}\int_\Omega |w|^p dx dt,\quad C_2=C_2(\Omega,n,p,s).
$$\end{prop} 

\begin{proof} We perform the proof by obtaining a uniform estimate for $w_M$ (see Section \ref{sec:not}). The proof is split into several different steps. We first perform the proof in the case $sp<n$ and then comment on how it would change for $sp=n$.\\

\noindent {\bf Step 1:} We test the equation with $J_\alpha (w_M) \eta (t)$, where $\alpha\geq 2$ and $\eta$ is a smooth function such that $\eta(t)=0$ for $t\leq t_1$ and $\eta(t)=1$ for $t\geq t_2$. By Lemma \ref{lem:test}
\begin{equation}\label{eq:testa}
\begin{split}
&\frac{1}{\alpha}\int_{t_1}^{t_2} \int_\Omega |w_M|^\alpha\eta'(t) dx dt = \frac{1}{\alpha}\int_\Omega |w_M(x,t_2)|^\alpha dx\\& + \int_{t_1}^{t_2} \int_{\R^n}\int_{\R^n} \eta(t) \big(J_p(v(x)-v(y))-J_p(u(x)-u(y))\big)(J^M_\alpha(w(x))-J^M_\alpha(w(y)) d\mu dt.
\end{split}
\end{equation}
For the last term, we use Lemma \ref{lem:ineqM} to obtain the lower bound
\begin{equation}\label{eq:lbd}
\begin{split}
& \frac{(\alpha-1)}{3\cdot 2^{p-1}}\left(\frac{p}{\alpha-2+p}\right)^p\int_{t_1}^{t_2} \int_{\R^n}\int_{\R^n} \eta(t) 	\big| |w_M(x)|^\frac{\alpha-2}{p} w_M(x)-|w_M(y)|^\frac{\alpha-2}{p} w_M(y)\big|^p d\mu dt\\
&=\frac{(\alpha-1)}{3\cdot 2^{p-1}}\left(\frac{p}{\alpha-2+p}\right)^p\int_{t_1}^{t_2}\eta(t) [|w_M|^\frac{\alpha-2}{p}w_M ]^p_{W^{s,p}(\R^n)}	dt
\end{split}
\end{equation}
Combining \eqref{eq:testa} and \eqref{eq:lbd}, we obtain 
\[
\begin{split}
\frac{1}{\alpha}\int_\Omega |w_M(x,t_2)|^\alpha\ dx&+\frac{(\alpha-1)}{3\cdot 2^{p-1}}\left(\frac{p}{\alpha-2+p}\right)^p\int_{t_1}^{t_2}\eta(t) [|w_M|^\frac{\alpha-2}{p}w_M ]^p_{W^{s,p}(\R^n)}	dt\\
&\leq \frac{1}{\alpha}\int_{t_1}^{t_2} \int_\Omega |w_M|^\alpha\eta'(t) dx dt, 
\end{split}
\]
or
\[
\begin{split}
\int_\Omega |w_M(x,t_2)|^\alpha dx+\alpha^{2-p}C(p)\int_{t_1}^{t_2}\eta(t) [|w_M|^\frac{\alpha-2}{p}w ]^p_{W^{s,p}(\R^n)}	dt\leq \int_{t_1}^{t_2} \int_\Omega |w_M|^\alpha\eta'(t) dx dt, 
\end{split}
\]
where we have used that
$$
\alpha(\alpha-1)\left(\frac{p}{\alpha-2+p}\right)^p\geq \alpha^{2-p}C(p)>0
$$
for all $\alpha\geq 2$.\\

\noindent {\bf Step 2: } 
By varying $t_2$, we obtain (as long as $\eta(t)=1$ for $t\geq t^*$ and $\eta(t_1)=0$)
\begin{equation}\label{eq:tstar}
 \sup_{t\in [t^*,t_2]}\int_\Omega |w_M|^\alpha\eta dx+\alpha^{2-p}C(p)\int_{t^*}^{t_2}\eta(t) [|w_M|^\frac{\alpha-2}{p}w_M ]^p_{W^{s,p}(\R^n)}	dt\leq 2\int_{t_1}^{t_2} \int_\Omega |w_M|^\alpha\eta'(t) dx dt.
\end{equation}
\noindent {\bf Step 3:} 
The parabolic Sobolev embedding (Theorem \ref{thm:parsob}) with $\kappa=\frac{p\alpha s}{(p-2+\alpha)n}+1$ applied with  $f=|w_M|^\frac{\alpha-2}{p}w$ gives
\begin{equation}\label{eq:sobappl}
\begin{split}
\int_{t_1}^{t_2}\int_\Omega |w_M|^{p-2+\alpha(1+\frac{sp}{n})} dx dt&\leq C\int_{t_1}^{t_2} [|w_M|^\frac{\alpha-2}{p}w_M ]^p_{W^{s,p}(\R^n)}	dt\\
&\times \left(\sup_{[t_1,t_2]}\int_\Omega w_M^\alpha dx\right)^\frac{sp}{n},
\end{split}
\end{equation}
for any $t_2\geq t_1$ and where $C=C(n,p,s)$.\\
\noindent {\bf Step 4:} 
Fix $T>0$, let $T_k=T(1-2^{-k})$ for $k=1,2,3,\ldots$ and choose non-negative smooth functions $\eta_k$ such that
$$
\eta_k(t)=\begin{cases} 0 & t\leq T_k\\
1& t\geq T_{k+1}\\
|\eta^\prime_k|\leq \frac{2^{k+1}}{T}.
\end{cases}
$$
Then \eqref{eq:sobappl} combined with \eqref{eq:tstar} with $t^*=T_{k+1}$ and $t_1=T_k$  imply
\[
\begin{split}
&\int_{T_{k+1}}^{t_2}\int_\Omega |w_M|^{p-2+\alpha(1+\frac{sp}{n})} dx dt\leq C\int_{T_{k+1}}^{t_2} [|w_M|^\frac{\alpha-2}{p}w_M ]^p_{W^{s,p}(\R^n)}	dt \left(\sup_{[T_{k+1},t_2]}\int_\Omega w_M^\alpha dx\right)^\frac{sp}{n}\\
&\leq C\int_{T_{k+1}}^{t_2}\eta_k(t) [|w_M|^\frac{\alpha-2}{p}w_M]^p_{W^{s,p}(\R^n)}	dt \left(\sup_{[T_{k+1},t_2]}\int_\Omega \eta_k(t) w_M^\alpha dx\right)^\frac{sp}{n}\\
&\leq \alpha^{p-2}CC(p)\left(\int_{T_k}^{t_2}\int_\Omega |\eta_k'(t)\|w_M|^\alpha dx dt \right)^{1+\frac{sp}{n}}\\
&\leq \alpha^{p-2}C(s,p,n)\left(\frac{2^{k+1}}{T}\right)^{1+\frac{sp}{n}}\left(\int_{T_k}^{t_2}\int_\Omega |w_M|^\alpha dx dt \right)^{1+\frac{sp}{n}},
\end{split}
\]
where we have absorbed the two constants into one. With the notation $\beta = 1+\frac{sp}{n}$ we obtain for $\alpha\geq 2$
\[
\left(\int_{T_{k+1}}^{t_2}\int_\Omega |w_M|^{p-2+\alpha\beta} dx dt\right)^\frac{1}{\beta}\leq (\alpha^{p-2}C(s,p,n))^\frac{1}{\beta}\frac{2^{k+1}}{T}\int_{T_k}^{t_2}\int_\Omega |w_M|^\alpha dx dt .
\]
Starting with 
$$\alpha_1=p,\quad \alpha_2 = p-2+p\beta,\quad \alpha_3=p-2+\alpha_2\beta = (p-2)(1+\beta)+p\beta^2$$
and more generally 
$$
\alpha_{i+1} = p-2+\alpha_{i}\beta, \quad i=1,2,3,\ldots
$$
we obtain
\begin{equation}\label{eq:kest}
\begin{split}
&\left(\int_{T_{k+1}}^{t_2}\int_\Omega |w_M|^{(p-2)(1+\beta+\beta^2+\cdots+\beta^{k-1})+p\beta^k} dx dt\right)^\frac{1}{\beta^k}\\
&\leq (C(s,p,n))^{\frac{1}{\beta}+\frac{1}{\beta^2}+\cdots +\frac{1}{\beta^{k-1}}}
\frac{\beta^{\frac{1}{\beta}+\cdots +\frac{k}{\beta^{k}}}2^{2+\frac{3}{\beta}+\cdots +\frac{k+1}{\beta^{k-1}}}}{T^{1+\frac{1}{\beta}+\cdots +\frac{1}{\beta^{k-1}}}}\int_{T_1}^{t_2}\int_\Omega |w_M|^p dx dt .
\end{split}
\end{equation}
Here we have used that
\[
\begin{split}
((p-2)(1+\beta+\beta^2+\cdots+\beta^{k-1})+p\beta^k)/\beta_k &= \frac{\frac{(p-2)(\beta^k-1)}{1-\beta}}{\beta^k}+p\\
&\to \frac{(p-2)}{1-\beta}+p\\
&=\frac{n(p-2)}{sp}+p
\end{split}
\]
so that $\alpha_k = ((p-2)(1+\beta+\beta^2+\cdots+\beta^{k-1})+p\beta^k)\leq (n+1)\beta^k$, which implies
$$
(\alpha_k^{p-2})^\frac{1}{\beta}\leq \left(n^{p-2}\beta^k\right)^\frac{1}{\beta}.
$$
We also absorbed the part $n^{p-2}$ into the constant $C(s,p,n)$. Since $\beta>1$, the geometric sum satifies
$$
1+\frac{1}{\beta}+\ldots = \frac{n+sp}{sp}
$$
and the two sums
$$
\frac{1}{\beta}+\cdots +\frac{k}{\beta^{k}} ,\quad 2+\frac{3}{\beta}+\cdots +\frac{k+1}{\beta^{k-1}}
$$
are both convergent. We may therefore pass $k\to\infty$ in \eqref{eq:kest} and obtain
$$
\|w_M\|_{L^\infty([T	,t_2]\times \Omega)}^{\frac{n(p-2)}{sp}+p}\leq C_1(n,s,p)T^{-\frac{n+sp}{sp}}\int_{T/2}^{t_2}\int_\Omega |w_M|^p dx dt .
$$
Since the constant is independent of $M$ this implies the desired result for $w$.\\

\noindent {\bf The case $sp\geq n$.} In this case we would choose any $N$ satisfying $n\leq sp<N$ and perform the same proof as above but with $n$ replaced by $N$. The difference is here when we apply the Sobolev embedding (in this case Theorem \ref{thm:parsob2}), then the constant will depend also on $\Omega$, otherwise the proof is identical.

\end{proof}
\section{Proof of Theorem \ref{thm:main1}}
We are now ready to prove the first main theorem. 
\begin{proof} As before we let $w=u-v$. We split the proof into different cases, depending on whether $sp>n$ or not.

\noindent {\bf Case 1: $sp>n$.} Due to the the embedding from $W^{s,p}_0(\R^n)$ into $L^\infty$, this case is simpler. Recall the notation 
$$
I_2(t)=\int_\Omega (w(x,t))^2 dx.
$$
By using Lemma \ref{lem:test} with $w$ as test function, we obtain with the third part of Proposition \ref{prop:sobpoin1} for a.e. $t$
$$
I'_2(t)=-[w(\cdot,t)]_{W^{s,p}(\R^n)}^p \leq -C\|w(\cdot,t)\|^p_{L^\infty(\Omega)},\quad C=C(\Omega,n,p,s).
$$
Therefore, upon integration from $t/2$ to $t$
$$
\int^{t}_{t/2}C\|w(\cdot,\tau)\|^p_{L^\infty(\Omega)} d\tau \leq I_2(t/2)-I(t)\leq I_2(t/2).
$$
By Proposition \ref{Prop:Degenerate}, $t\mapsto\|w(\cdot,t)\|^p_{L^\infty(\Omega)}$ is non-increasing. Hence, 
$$
C(t-t/2)\|w(\cdot,t)\|^p_{L^\infty(\Omega)}\leq \int^{t}_{t/2}C\|w(\cdot,\tau)\|^p_{L^\infty(\Omega)} d\tau\leq I_2(t/2).
$$
Using Proposition \ref{Prop:Degenerate} for $m=2$, we arrive at
$$
\|w(\cdot,t)\|^p_{L^\infty(\Omega)}\leq \frac{C}{t} I_2(t/2)\leq Ct^{-\frac{p}{p-2}}, \quad C=C(\Omega,n,p,s)
$$
which implies the desired result.

\noindent {\bf Case 2: $sp< n$. }
By Proposition \ref{prop:degmoser} (recall  $t\mapsto\|w(\cdot,t)\|_{L^\infty(\Omega)}$ is non-increasing)
\begin{equation}
\label{eq:esti1}
\|w\|_{L^\infty([T	,\infty)]\times \Omega)}^{\frac{n(p-2)}{sp}+p}\leq CT^{-\frac{n+sp}{sp}}\int_{T/2}^{T}\int_\Omega |w|^p dx dt ,\quad C=C(n,s,p).
\end{equation}
By Proposition \ref{Prop:Degenerate} 
$$
T^{-\frac{n+sp}{sp}}\int_{T/2}^{T}\int_\Omega |w|^p dx\leq CT^{-\frac{n+sp}{sp}}\int_{T/2}^{T} t^{-\frac{1}{p-2}} dt\leq CT^{-\frac{n}{sp}-\frac{p}{p-2}},\quad C=C(\Omega,n,p,s).
$$
Inserting this into \eqref{eq:esti1} gives
$$
\|w\|_{L^\infty([T	,\infty)]\times \Omega)}\leq Ct^{-\frac{1}{p-2}}, \quad C=C(\Omega,n,p,s).
$$

\noindent {\bf Case 3: $sp=n$. } Proceed as in the case above, with $n$ replaced by $N$ so that $N>sp$.
\end{proof}

\section{The singular case: $\max\{1, \frac{2n}{n+2s}\}<p<2$}

\subsection{A first decay estimate}
Below we study the decay of $L^p$-norms in the case $n\geq 2$. For the case $n=1$, we refer to Remark \ref{re:singdecay}.

\begin{prop}\label{Prop:Singular} Let $n\geq 2$, $\frac{2n}{n+2s}\leq p<2$, $\Omega$ be a bounded domain and let $u$ and $v$ be weak solutions of 
$$
u_t+(-\lap_p)^s u = f\quad \text{ in } \Omega\times (0,\infty).$$
Assume in addition that $f\in L^{p^\prime}_{\text{loc}}((0,\infty);(X^{s,p}_0(\Omega,\Omega^\prime))^*)$,  $u,v\in L^\infty_{\text{loc}}((0,\infty);W^{s,p}(\R^n))$ and $w=u-v\in L^p_{\text{loc}}((0,\infty);W_0^{s,p}(\Omega))$. Then with
$$
I_m(t)=\int_\Omega |w(x,t)|^m dx, \quad m\geq 2.
$$
there holds
$$
I_m(t_2)\leq I_m(t_1)e^{-\frac{CL(m-1)}{m}(t_2-t_1)},
$$
for $t_2\ge t_1 >0$ and where $C=C(\Omega,n,p,s)$ and 
$$
L=\left[\mathop{\sup}\limits_{t_1\leq \tau\leq t_2}([v(\cdot,\tau)]_{W^{s,p}(\R^n)}+[u(\cdot,\tau)]_{W^{s,p}(\R^n)})\right]^{p-2}.
$$
\end{prop}

\begin{proof} As in the proof of Proposition \ref{Prop:Degenerate}, we will prove the desired estimate for 
$$
I_m^M(t)=\int_\Omega |w_M(x,t)|^m dx,
$$
with constants independent of $M$. This will imply the desired result for $w$. By Lemma \ref{lem:test}
\begin{align}
-m\int^{T_1}_{T_0}\iint_{\R^{n}\times\R^{n}}&\Big(J_p(u(x,t)-u(y,t))-J_p(v(x,t)-v(y,t))\Big)\nonumber\\
&\times\Big(J^M_{m}(u(x,t)-v(x,t))-J^M_{m}(u(y,t)-v(y,t))\Big)d\mu dt\label{mdiff2} \\
&=I^M_m(T_1)-I^M_m(T_0).\nonumber 
\end{align}
Using Lemma \ref{lemma:Msing_ineq_2}, we obtain the following estimate of the integrand in the left hand side of \eqref{mdiff2}
\begin{equation}
\label{eq:pointsing}
\begin{split}
&\big(J^M_m(w(x))-J^M_m(w(y))\big)\big(J_p(u(x)-u(y))-J_p(v(x)-v(y))\big)\\
&\geq \frac{4(m-1)(p-1)}{m^2} \left|w_M(x)^\frac{m-2}{2}w_M(x)-w_M(y)^\frac{m-2}{2}w_M(y)\right|^2 \big(|v(x)-v(y)|+|u(x)-u(y)|\big)^{p-2}.
\end{split}
\end{equation}
In addition, H\"older's inequality with exponents $2/p$ and $2/(2-p)$ implies for functions $a,b,c$ and $d$
\[
\begin{split}
&\iint_{\R^{n}\times\R^{n}} |a^\frac{m-2}{2}a-b^\frac{m-2}{2}b|^p d\mu = \iint_{\R^{n}\times\R^{n}} |a^\frac{m-2}{2}a-b^\frac{m-2}{2}b|^p (c+d)^{(p-2)\frac{p}{2}} (c+d)^{(2-p)\frac{p}{2}} d\mu  \\
&\leq \left(\iint_{\R^{n}\times\R^{n}} |a^\frac{m-2}{2}a-b^\frac{m-2}{2}b|^2 (c+d)^{(p-2)}d\mu \right)^\frac{p}{2}\left(\iint_{\R^{n}\times\R^{n}} |c+d|^p d\mu\right)^\frac{2-p}{2}.
\end{split}
\]
By replacing $a, b, c, d$ with $w_M(x), w_M(y), |v(x)-v(y)|, |u(x)-u(y)|$ respectively, this implies
\begin{equation}
\label{eq:longineq}
\begin{split}
\iint_{\R^{n}\times\R^{n}} &|w_M(x)^\frac{m-2}{2}w(x)-w(y)^\frac{m-2}{2}w_M(y)|^2 (|v(x)-v(y)|+|u(x)-u(y)|)^{(p-2)} d\mu\\
&\geq \left(\iint_{\R^{n}\times\R^{n}} |w_M(x)^\frac{m-2}{2}w_M(x)-w_M(y)^\frac{m-2}{2}w_M(y)|^p d\mu\right)^\frac{2}{p}\\
&\times \left(\iint_{\R^{n}\times\R^{n}} (|v(x)-v(y)|+|u(x)-u(y)|)^pd\mu\right)^\frac{p-2}{p}\\
&\geq C\left(\int_\Omega |w_M|^\frac{mp^*_s}{2}dx\right)^\frac{2}{p^*_s}\left(\iint_{\R^{N}\times\R^{N}} (|v(x)-v(y)|+|u(x)-u(y)|)^p d\mu\right)^\frac{p-2}{p}\\
&\geq C\int_\Omega |w_M|^mdx\left(\iint_{\R^{n}\times\R^{n}} (|v(x)-v(y)|+|u(x)-u(y)|)^p d\mu\right)^\frac{p-2}{p}\\
&\geq C\int_\Omega |w_M|^mdx \left([v]_{W^{s,p}(\R^n)}+[u]_{W^{s,p}(\R^n)}\right)^{p-2},\quad C=C(\Omega,n,s,p),
\end{split}
\end{equation}
for a.e. $t$. Here we have used the Sobolev embedding (Theorem \ref{thm:sob}) together with the observation that $p_s^*=np/(n-sp)$ so that $p^*_s\geq 2$ since $p\geq 2n/(n+2s)$. Integrating \eqref{eq:longineq} and using \eqref{eq:pointsing} together with \eqref{mdiff2}, we obtain
$$
I^M_m(T_1)-I^M_m(T_0)\leq -C\frac{m-1}{m}\int_{T_0}^{T_1}\int_\Omega |w_M(\cdot,t)|^mdx\left([v(\cdot,t)]_{W^{s,p}(\R^n)}+[u(\cdot,t)]_{W^{s,p}(\R^n)}\right)^{p-2} dt,$$
where $C=C(\Omega,n,s,p)$. Therefore, for a.e. $t$, we have
$$
(I^M)'_m(t)\leq -C\frac{m-1}{m} I^M_m ([v(\cdot,t)]_{W^{s,p}(\R^n)}+[u(\cdot,t]_{W^{s,p}(\R^n)})^{p-2}.
$$
Arguing as in the proof of Proposition \ref{Prop:Degenerate}, this implies the desired result for $w_M$ which ends the proof.

\end{proof}

\begin{rem}\label{re:singdecay}
For $n=1$ and $\max\{1,\frac{2}{1+2s}\}<p<2$, we have two different cases: $sp<n=1$ or $sp\geq n=1$ . In the case $sp<n=1$, the proof follows as the proof above. 
In the other case, when $sp\geq n=1$, we simply replace $n$ in $p^*_s$ by any $N>sp$ in the proof of Proposition \ref{Prop:Singular}, and then use the second and third inequalities in Proposition \ref{prop:sobpoin1}.
\end{rem}

\subsection{The $L^\infty$-estimate via Moser iteration} 

We perform a Moser-type iteration in order to obtain an $L^\infty$-estimate for the difference of two solutions.

\begin{prop} \label{prop:singmoser} Suppose $\max\{1,\frac{2n}{n+2s}\}<p<2$ and $t_2\in (T,\frac{3T}{2})$ for some $T>0$. Let $\Omega$ be a bounded domain and let $u$ and $v$ be weak solutions of 
$$
u_t+(-\lap_p)^s u = f\quad \text{ in } \Omega\times (0,\infty).$$
Assume in addition that $f\in L^{p^\prime}_{\text{loc}}((0,\infty);(X^{s,p}_0(\Omega,\Omega^\prime))^*)$,  $u,v\in L^\infty_{\text{loc}}((0,\infty);W^{s,p}(\R^n))$ and $w=u-v\in L^p_{\text{loc}}((0,\infty);W_0^{s,p}(\Omega))$. If $sp<n$ then
$$
\mathop{\sup}\limits_{T\leq t\leq t_2}\|w\|^2_{L^{\infty}(\Omega\times\{t\})}\leq
C_1\frac{T^\frac{(2-p)\nu-2p}{2(p-\nu)}}{L^\frac{p\nu}{2(p-\nu)}}\int_{\frac{T}{2}}^{t_2}\int_\Omega |w|^2 dxdt,\quad C_1=C_1(n,p,s),
$$
where $\nu=\frac{2n}{n+2s}$.

If $n\leq sp< N$ then
$$
\mathop{\sup}\limits_{T\leq t\leq t_2}\|w\|^2_{L^{\infty}(\Omega\times\{t\})}\leq
C_2\frac{T^\frac{(2-p)\nu-2p}{2(p-\nu)}}{L^\frac{p\nu}{2(p-\nu)}}\int_{\frac{T}{2}}^{t_2}\int_\Omega |w|^2 dxdt,\quad C_2=C_2(\Omega,n,p,s),
$$
where $\nu=\frac{2N}{N+2s}$.

In both cases above $L=\left[\mathop{\sup}\limits_{\frac{T}{2}\leq \tau\leq t_2}\left([v(\cdot,\tau)]_{W^{s,p}(\R^n)}+[u(\cdot,\tau)]_{W^{s,p}(\R^n)}\right)\right]^{p-2}$.

\end{prop}

\begin{proof}
As in the proof of Proposition \ref{prop:degmoser}, we perform the proof for $w_M$ and split the proof into two cases: $sp<n$ and $sp\geq n$.

\bigskip

\noindent {\bf Case 1: $sp<n$.}

{\bf Step 1:} As in step 1 of the case $p\geq 2$ we test the equation with the function $J_\alpha (w_M) \eta (t)$, where $\eta$ is a smooth function such that $\eta(t)=0$ for $t\leq t_1$ and $\eta(t)=1$ for $t\geq t_2$. By Lemma \ref{lem:test}
\begin{equation}\label{sing1}
\begin{split}
&\frac{1}{\alpha}\int_{t_1}^{t_2} \int_\Omega |w_M|^\alpha\eta'(t) dx dt = \frac{1}{\alpha}\int_\Omega |w_M(x,t_2)|^\alpha dx\\& + \int_{t_1}^{t_2} \int_{\R^n}\int_{\R^n} \eta(t) (J_p(v(x)-v(y))-J_p(u(x)-u(y)))(J^M_\alpha(w(x))-J^M_\alpha(w(y)) d\mu dt,
\end{split}
\end{equation}
Using Lemma \ref{lemma:sing_ineq_2} as is done in the proof of Proposition \ref{Prop:Singular}, the last term can be bounded from below by
\begin{equation}\label{sing2}
\begin{split}
&C\frac{\alpha-1}{\alpha^2}\int_{t_1}^{t_2} [w_M^\frac{\alpha-2}{2}w_M]^2_{W^{s,p}(\mathbb{R}^n)} ([v]_{W^{s,p}(\mathbb{R}^n)}+[u]_{W^{s,p}(\mathbb{R}^n)})^{p-2} \eta(t) dt\\
&\geq C\frac{\alpha-1}{\alpha^2} \left(\mathop{\sup}\limits_{t_1\leq t\leq t_2}\left([v(\cdot,t)]_{W^{s,p}(\mathbb{R}^n)}+[u(\cdot,t)]_{W^{s,p}(\mathbb{R}^n)}\right)\right)^{p-2}
\int_{t_1}^{t_2} [w_M^\frac{\alpha-2}{2}w_M]^2_{W^{s,p}(\R^n)}  \eta(t) dt,
\end{split}
\end{equation}
where $C=C(p)$. By H\"older's inequality as above we have 
\begin{equation}\label{sing3}
\begin{split}
\int_{t_1}^{t_2}\eta(t) [|w_M|^\frac{\alpha-2}{2}w_M]^2_{W^{s,p}(\R^n)} dt \geq 
\left(\int_{t_1}^{t_2}\eta(t) [|w_M|^\frac{\alpha-2}{2}w_M]^p_{W^{s,p}(\R^n)} dt\right)^\frac{2}{p}
\left(\int_{t_1}^{t_2}\eta(t)  dt\right)^\frac{p-2}{p}.
\end{split}
\end{equation}
Combining \eqref{sing1}, \eqref{sing2} and \eqref{sing3} we obtain
\[
\begin{split}
\frac{1}{\alpha}\int_\Omega |w(x,t_2)|^\alpha dx &+
C\frac{\alpha-1}{\alpha^2} L(t_1, t_2) \left(\int_{t_1}^{t_2}\eta(t) [|w_M|^\frac{\alpha-2}{2}w ]^p_{W^{s,p}(\R^n)} dt\right)^\frac{2}{p}\left(\int_{t_1}^{t_2}\eta(t)  dt\right)^\frac{p-2}{p}\\
&\leq \frac{1}{\alpha}\int_{t_1}^{t_2} \int_\Omega |w_M|^\alpha\eta'(t) dx dt, 
\end{split}
\]
where $L(t_1, t_2)$ is given by
$$
L(t_1, t_2) = \left(\mathop{\sup}\limits_{t_1\leq t\leq t_2}\left([v(\cdot,t)]_{W^{s,p}(\mathbb{R}^n)}+[u(\cdot,t)]_{W^{s,p}(\mathbb{R}^n)}\right)\right)^{p-2}.
$$

{\bf Step 2:} By varying $t_2$, we obtain (as long as $\eta(t)=1$ for $t\geq t^*$ and $\eta(t_1)=0$)
\begin{equation}
\label{eq:singstar}
\begin{split}
\sup_{t\in [t^*,t_2]}\int_\Omega |w_M|^\alpha\eta dx &+
C\frac{\alpha-1}{\alpha} L (t_1,t_2)\left(\int_{t^*}^{t_2}\eta(t) [|w_M|^\frac{\alpha-2}{2}w_M]^p_{W^{s,p}(\R^n)} dt\right)^\frac{2}{p}\left(\int_{t^*}^{t_2}\eta(t)  dt\right)^\frac{p-2}{p}\\
&\leq 2\int_{t_1}^{t_2} \int_\Omega |w_M|^\alpha\eta'(t) dx dt, \quad C=C(p).
\end{split}
\end{equation}

{\bf Step 3:} By the parabolic Sobolev inequality (Theorem \ref{thm:parsob}), applied to $f=|w_M|^\frac{\alpha-2}{2}w_M$ with $\kappa=\frac{2s}{n}+1$, we have
\begin{equation}
\label{eq:singstar2}
\begin{split}
\int_{t_1}^{t_2}\int_\Omega |w_M|^{p\alpha \frac{2s+n}{2n}} dx dt&\leq C\int_{t_1}^{t_2}\eta(t) [|w_M|^\frac{\alpha-2}{2}w_M]^p_{W^{s,p}(\R^n)}dt\\
&\times \left(\sup_{[t_1,t_2]}\int_\Omega |w_M|^\alpha dx\right)^\frac{sp}{n},\quad C=C(n,p,s).
\end{split}
\end{equation}

{\bf Step 4:} Let $t_2\in(T,\frac{3}{2}T)$ for $T>0$. Define as in the degenerate case $T_k=T(1-2^{-k})$ for $k=1,2,3,\ldots$ and the functions 
$$
\eta_k(t)=\begin{cases} 0 & t\leq T_k\\
1& t\geq T_{k+1}\\
|\eta^\prime_k|\leq \frac{2^{k+1}}{T}.
\end{cases}
$$

Then \eqref{eq:singstar} and \eqref{eq:singstar2} with $t^*=T_{k+1}$ and $t_1=T_k$ imply
\[
\begin{split}
&\int_{T_{k+1}}^{t_2}\int_\Omega |w_M|^{p\alpha \frac{2s+n}{2n}} dx dt\leq C\int_{T_{k+1}}^{t_2} [|w_M|^\frac{\alpha-2}{2}w_M]^p_{W^{s,p}(\R^n)}	dt
\times \left(\sup_{[T_{k+1},t_2]}\int_\Omega |w_M|^\alpha dx\right)^\frac{sp}{n}\\
&\leq C\int_{T_{k+1}}^{t_2} \eta_k(t)[|w_M|^\frac{\alpha-2}{2}w_M]^p_{W^{s,p}(\R^n)}	dt
\times \left(\sup_{[T_{k+1},t_2]}\int_\Omega \eta_k(t)|w_M|^\alpha dx\right)^\frac{sp}{n}\\
&\leq \left(C\frac{\alpha}{L(T_k,t_2)(\alpha-1)}\right)^\frac{p}{2} \left(\int_{T_k}^{t_2} \int_\Omega |w_M|^\alpha\eta'(t) dx dt\right)^{\left(\frac{sp}{n}+\frac{p}{2}\right)}\left(\int_{t^*}^{t_2}\eta(t)  dt\right)^\frac{2-p}{2}\\
&\leq \left(C\frac{\alpha}{L(T_k,t_2)(\alpha-1)}\right)^\frac{p}{2} T^\frac{2-p}{2}\left(\int_{T_k}^{t_2} \int_\Omega |w_M|^\alpha\eta'(t) dx dt\right)^{\left(\frac{sp}{n}+\frac{p}{2}\right)}\\
&\leq \left(C\frac{\alpha}{L(T_k,t_2)(\alpha-1)}\right)^\frac{p}{2}T^\frac{2-p}{2}\left(\frac{2^{k+1}}{T}\right)^{\left(\frac{sp}{n}+\frac{p}{2}\right)} \left(\int_{T_k}^{t_2} \int_\Omega |w_M|^\alpha dx dt\right)^{\left(\frac{sp}{n}+\frac{p}{2}\right)},\quad C=C(n,p,s).
\end{split}
\] 
Let $\beta=\frac{sp}{n}+\frac{p}{2}$. Note that $\beta>1$ since $\frac{2n}{n+2s}<p$. Then the above estimate can be written as
$$
\left(\int_{T_{k+1}}^{t_2}\int_\Omega |w_M|^{\alpha\beta} dx dt\right)^\frac{1}{\beta}\leq \left(C\frac{\alpha}{L(\alpha-1)}\right)^\frac{p}{2\beta}T^\frac{2-p}{2\beta}\frac{2^{k+1}}{T}\int_{T_k}^{t_2} \int_\Omega |w_M|^\alpha dx dt,\quad C=C(n,p,s).
$$
Here we have replaced $L(T_k,2)$ with $L=L(T/2,t_2)$. Now we wish to iterate this, starting from $\alpha_1=2$ and letting $\alpha_k = \beta\alpha_{k-1}$. Then we note that $\alpha/(\alpha-1)\leq 2$ and then we can simply write
$$
\left(C\frac{\alpha}{(\alpha-1)}\right)^\frac{p}{2\beta}\leq C^\frac{1}{\beta},\quad C= C(n,s,p),
$$
upon renaming the constant. With this in mind and the notation $\gamma=1+\frac{2s}{n}$, the estimate becomes
$$
\left(\int_{T_{k+1}}^{t_2}\int_\Omega |w_M|^{\alpha\beta} dx dt\right)^\frac{1}{\beta}\leq C^\frac{1}{\beta}\frac{1}{L^\frac{1}{\nu}}T^\frac{2-p}{2\beta}\frac{2^{k+1}}{T}\int_{T_k}^{t_2} \int_\Omega |w_M|^\alpha dx dt,\quad C=C(n,p,s).
$$
The first iteration becomes
\begin{equation}\label{ineq:sing2}
\begin{split}
\left(\int_{T_2}^{t_2}\int_\Omega |w_M|^{2\beta} dx dt\right)^\frac{1}{\beta}&\leq 
\frac{C^\frac{1}{\beta}}{L^\frac{1}{\gamma}}\frac{4}{T}T^\frac{2-p}{2\beta}
\int_{T_1}^{t_2}\int_\Omega |w_M|^2 dxdt,
\end{split}
\end{equation}
and the second one
\[
\begin{split}
\left(\int_{T_3}^{t_2}\int_{\Omega}|w_M|^{2\beta^2}\right)^\frac{1}{\beta^2}
&\leq \frac{C^{\frac{1}{\beta}+\frac{1}{\beta^2}}}{L^{\frac{1}{\gamma}(1+\frac{1}{\beta})}}
\frac{2^{2+\frac{3}{\beta}}}{T^{1+\frac{1}{\beta}}}T^{\frac{2-p}{2}(\frac{1}{\beta}+\frac{1}{\beta^2})}
\int_{T_1}^{t_2}\int_\Omega |w_M|^2 dxdt.
\end{split}
\]
Note that $$1+\frac{1}{\beta}+\frac{1}{\beta^2}+\ldots=\frac{\beta}{\beta-1}=\frac{p}{p-\nu},$$
$$\frac{1}{\beta}+\frac{1}{\beta^2}+\ldots=\frac{\nu}{p-\nu},$$ where $\nu=\frac{2n}{n+2s}$,
and the sum
$$
2+\frac{3}{\beta}+\frac{4}{\beta^2}+\ldots
$$
is convergent since $\beta>1$. Using this and continuing the iteration, we arrive at the estimate
$$
\mathop{\sup}\limits_{T\leq t\leq t_2}\|w_M\|^2_{L^{\infty}(\Omega\times\{t\})}\leq
C^\prime\frac{T^\frac{(2-p)\nu-2p}{2(p-\nu)}}{L^\frac{p\nu}{2(p-\nu)}}\int_{\frac{T}{2}}^{t_2}\int_\Omega |w_M|^2 dxdt,
$$ 
where $C=C(n,s,p)$, upon renaming the constant again. Since the constant is independent of $M$ this implies the desired result for $w$.\\

\bigskip

\noindent {\bf Case 2: $sp\geq n$.} In this case we would choose any $N$ such that $n\leq sp<N$ and perform the same proof as above but with $n$ replaced by $N$. The difference is here when we apply the Sobolev embedding, then the constant will depend also on $\Omega$, otherwise the proof is identical.

\end{proof}

\section{Proof of Theorem \ref{thm:main2}}

\begin{proof} As in the proof of Theorem \ref{thm:main1}, we split the proof into different cases, depending on whether $sp<n$ or not. Let $w=u-v$, $t>2\tau>0$. 

\noindent {\bf Case 1: $sp< n$. }
Let $t>2\tau>0$. By Proposition \ref{prop:singmoser} together with the fact that $(2-p)\nu-2p<0$  and that $\|w(\cdot,t)\|_{L^2(\Omega)}$ is non-increasing as a function of $t$ we obtain
\begin{equation}
\label{ineq:sing3}
\|w\|^2_{L^{\infty}(\Omega\times\{[t,\infty)\})}\leq
C_1\frac{t^\frac{(2-p)\nu-2p}{2(p-\nu)}}{L^\frac{p\nu}{2(p-\nu)}}\int_{\frac{t}{2}}^{t}\int_\Omega |w|^2 dxdt,\quad C_1=C_1(n,p,s), \quad \nu=\frac{2n}{n+2s}.
\end{equation}
Here $L=L(t,\infty)$ as in the statement of the Theorem 2. By Proposition \ref{Prop:Singular} 
$$
\int_\Omega |w(x,t_2)|^2 dx\leq e^{-\lambda L(t_2-t_1)} \int_\Omega |w(x,t_1)|^2 dx
$$
for $t_2\geq t_1>0$, where $\lambda=\lambda(\Omega,n,p,s)$ and as before 
$$
L(t_1,t_2)=\left[\mathop{\sup}\limits_{t_1\leq \tau\leq t_2}([v(\cdot,\tau)]_{W^{s,p}(\R^n)}+[u(\cdot,\tau)]_{W^{s,p}(\R^n)})\right]^{p-2}.
$$
Since $\|w(\cdot,t)\|_{L^2(\Omega)}$ is non-increasing as a function of $t$, we have
\[
\begin{split}
\int^{t}_{\frac{t}{2}}\int_{\Omega}|w(x,\ell)|^2dxd\ell&\leq\int_{\Omega}|w(x,\tau)|^2dx\int^{t}_{\frac{t}{2}}e^{-\lambda L(\ell-\tau)}d\ell\\
&=\frac{1}{-\lambda L}\left(e^{-\lambda L(t-\tau)}-e^{-\lambda L(\frac{t}{2}-\tau)}\right)\int_{\Omega}|w(x,\tau)|^2dx\\
&\leq \frac{1}{\lambda L}e^{-\lambda L(\frac{t}{2}-\tau)}\int_{\Omega}|w(x,\tau)|^2dx.
\end{split}
\]
Here we have replaced $L(t_1,t_2)$ with $L=L(t,\infty)$. Combined with \eqref{ineq:sing3}, this yields
\[
\begin{split}
\|w\|^2_{L^{\infty}(\Omega\times\{[t,\infty)\})}&\leq
C\frac{t^{\frac{(2-p)\nu-2p}{2(p-\nu)}}}{L^{\frac{p\nu}{2(p-\nu)}+1}}e^{-\lambda L(\frac{t}{2}-\tau)}\int_{\Omega}|w(x,\tau)|^2dx\\
&=C\frac{1}{(tL)^{\frac{p\nu}{2(p-\nu)}+1}}e^{-\lambda L(\frac{t}{2}-\tau)}\int_{\Omega}|w(x,\tau)|^2dx,\quad C=C(\Omega,n,p,s).
\end{split}
\]
Hence,
$$
\|w\|_{L^{\infty}(\Omega\times\{[t,\infty)\})}\leq C(tL)^{-\frac{p\nu}{4(p-\nu)}-\frac{1}{2}}e^{-\lambda L(\frac{t}{2}-\tau)}\|w\|_{L^2(\Omega\times\{\tau\})},
$$
where $C=C(\Omega,n,s,p)$ and $\lambda=\lambda(\Omega,n,s,p)$. This is the desired result.

\medskip

\noindent {\bf Case 2: $n\leq sp< N$. } Proceed as in the case above, with $n$ replaced by $N$.
\end{proof}

\section{Examples}\label{sec:ex}
In this section, we provide examples showing that the rates of convergence obtained in Theorem 1 and 2 are sharp. In the case of Theorem 2, we have not been able to provide an example showing that the decay rate is sharp when $f=0$. 

\begin{ex}[Sharpness of Theorem 1]
In \cite{Vaz15}, it is proved that there are solutions of \eqref{MainPDE} with $f=0$ of the form
$$
U(x,t)=t^\frac{-1}{p-2}F(x), 
$$
where $F$ solves the equation
$$
(-\lap_p)^s F = F.
$$
This shows that the decay to the stationary solution which is identically zero is not faster than of order $t^{-1/(p-2)}$ so that the rate obtained in Theorem 1 is sharp.
\end{ex}

\begin{ex}[Sharpness of Theorem 2] We prove that there is a solution that does not converge faster than exponentially to the stationary solution, so that the rate cannot be better in general. It should be mentioned that the same separation of variables as in the example above yields a solution that has \emph{extinction in finite time}. See Theorem 5.2 in \cite{abd} and Section 7 in \cite{Vaz15}.

Let $u$ be a solution of 
$$
\begin{cases}
(-\lap_p)^s u = u& \text{ in } \Omega\\
u = 0 & \text{ in }\Omega^c
\end{cases}
$$
for a bounded domain $\Omega$. By Theorem 3.1 in \cite{ILPS16}, $u\in L^\infty(\Omega)$. This solution can be constructed as a minimizer of the seminorm $[u]_{W^{s,p}(\R^n)}$ over functions vanishing outside $\Omega$ with $L^2$-norm equal to 1, given that $2>p> 2n/(n+2s)$, so that the embedding of $W^{s,p}$ into $L^2$ is compact.

For $\lambda>0$ let 
$$
w=(1+e^{-\lambda t})u.
$$
We will now show that $w$ satisfies
$$
\begin{cases}
w_t+(-\lap_p)^s w \leq u& \text{ in }\Omega\times (0,\infty)\\
w=2u &\text{ in } \Omega\times\{0\}\\
w = 0 & \text{ in }\Omega^c.
\end{cases}
$$
The second and the third equalities are trivial. To show the first one we note that 
$w_t=-\lambda e^{-\lambda t}u$ and 
\[
\begin{split}
(-\lap_p)^s w & = (1+e^{-\lambda t})^{p-1}(-\lap_p)^s u \\
&=(1+ e^{-\lambda t})^{p-1}u\\
&\leq\left(1+ c(1+e^{-\lambda t})^{p-2}e^{-\lambda t}\right) u\\
&\leq u+ce^{-\lambda t} u,
\end{split}
\]
where we have used Lemma \ref{lemma:simple inequality} to estimate $(1+\lambda e^{-\lambda t})^{p-1}$. Therefore
$$
w_t+(-\lap_p)^s w\leq -\lambda e^{-\lambda t}u+u+ce^{-\lambda t} u=u+e^{-\lambda t}u(-\lambda+c)\leq u
$$
if $\lambda\geq c$. Hence $w$ satisifies the desired properties. Take now $v$ to be a solution of (provided by Theorem 3.3 in \cite{GT21}). 
$$
\begin{cases}
v_t+(-\lap_p)^s v = u& \text{ in }\Omega\times (0,\infty)\\
v=2u&\text{ in } \Omega\times\{0\}\\
v = 0 &\text{ in } \Omega^c.
\end{cases}
$$
Then by Theorem 2, $v\to u$ exponentially and $v\geq w$ by comparison (Proposition \ref{prop:comp2}). Therefore, 
$$
v-u\geq w-u=e^{-\lambda t} u.
$$
Hence, $v$ will not converge faster than exponentially to $u$.
\end{ex}

\appendix

\section{Comparison principles and regularization of test functions}
\subsection{Testing the equation with powers of $u$ and $v$}
In this section we prove that we may test the equation with Lipschitz functions of $u-v$.

The result below is very similar to Lemma $3.3$ in \cite{BLS21}. For completeness, we spell out some details here. In order to perform the proof we need to regularize the test function in time as is done in \cite{BLS21}. Let $\zeta: \mathbb{R}\rightarrow\mathbb{R}$ be a nonnegative, even smooth function with compact support in $(-1/2,1/2)$, satisfying $\int_{\mathbb{R}}\zeta(\tau)d\tau=1$. If $f\in L^1((a,b))$, we define the convolution
\begin{equation}\label{convolution}
f^{\epsilon}(t) = \frac{1}{\e}\int^{t+\frac{\e}{2}}_{t-\frac{\e}{2}}\zeta(\frac{t-\ell}{\e})f(\ell)d\ell=\frac{1}{\e}\int^{\frac{\e}{2}}_{-\frac{\e}{2}}\zeta(\frac{\sigma}{\e})f(t-\sigma)d\sigma, \text{\quad for }t\in(a,b),
\end{equation} 
where $0<\e<\min\{b-t, t-a\}$.

\begin{lem}\label{lem:test}
Assume that $u$ and $v$ are weak solutions of \eqref{locweaksol} in $\Omega\times (0,\infty)$ with $f\in L^{p^\prime}_{\rm loc}((0,\infty);(X^{s,p}_0(\Omega,\Omega^\prime))^*)$ such that $u-v\in L^p_{loc}((0,\infty); W_0^{s,p}(\Omega))$. Let $\eta:\R\to\R$ be a smooth function.
Then, for any globally Lipschitz function $F(t)$ $(t\in\R)$ such that $F(0)=0$ we have
\begin{align*}\label{test_func}
\int^{T_1}_{T_0}\iint_{\R^{N}\times\R^{N}}&\Big(J_p\big(u(x,t)-u(y,t)\big)-J_p\big(v(x,t)-v(y,t)\big)\Big)\\
&\times\Big(F\big(u(x,t)-v(x,t)\big)-F\big(u(y,t)-v(y,t)\big)\Big)\eta(t)d\mu dt\\
&+\int_{\Omega}\FF\big(u(x,T_1)-v(x,T_1)\big)\eta(T_1)dx\\
&=\int_{\Omega}\FF\big(u(x,T_0)-v(x,T_0)\big)\eta(T_0)dx+\int^{T_1}_{T_0}\int_{\Omega}\FF\big(u(x,t)-v(x,t)\big)\eta'(t) dx dt,
\end{align*}
where $\FF(t) = \int^{t}_0F(\rho)d\rho$.
\end{lem}
\begin{proof}
Let $(T_0, T_1)\Subset J \subset(0,\infty)$ and $\phi\in L^{p}(J;X^{s,p}_0(\Omega,\Omega^\prime))\cap C^1(J;L^2(\Omega))$. For $\e$ small enough,
we use the time-regularized function $\phi^\e$ as a test function in \eqref{locweaksol}. 
By the properties of convolution, Fubini's theorem and integration by parts we obtain
\[
\begin{split}
-\int^{T_1}_{T_0}\int_{\Omega}u(x,t)\partial_t\phi^\e(x,t)&dxdt=\int_{\Omega}\int^{T_1-\frac{\e}{2}}_{T_0+\frac{\e}{2}}\partial_{\ell}u^{\e}(x, \ell)\phi(x,\ell)d\ell dx + \Sigma(u,\e)\\
&-\int_{\Omega}\Big[u^{\e}\big(x,T_1-\frac{\e}{2}\big)\phi\big(x, T_1-\frac{\e}{2}\big)-u^{\e}\big(x,T_0+\frac{\e}{2}\big)\phi\big(x, T_0+\frac{\e}{2}\big)\Big]dx,
\end{split}
\]
where
\[
\begin{split}
\Sigma(u,\e)=&-\int_{\Omega}\int^{T_0+\frac{\e}{2}}_{T_0-\frac{\e}{2}}\Big(\frac{1}{\e}\int^{\ell+\frac{\e}{2}}_{T_0}u(x,t)\zeta\big(\frac{\ell-t}{\e}\big)dt\Big)\partial_{\ell}\phi(x,\ell)d\ell dx\\
&-\int_{\Omega}\int^{T_1+\frac{\e}{2}}_{T_1-\frac{\e}{2}}\Big(\frac{1}{\e}\int^{T_1}_{\ell-\frac{\e}{2}}u(x,t)\zeta\big(\frac{\ell-t}{\e}\big)dt\Big)\partial_{\ell}\phi(x,\ell)d\ell dx.
\end{split}
\]
Since $u$ is a weak solution, it follows that for $0<\e<\e_0$
\begin{equation}\label{test_u}
\begin{split}
\int^{T_1}_{T_0}\iint_{\R^n\times\R^n}&J_p\big(u(x,t)-u(y,t)\big)\big(\phi^{\e}(x,t)-\phi^{\e}(y,t)\big)d\mu dt\\
&+\int_{\Omega}\int^{T_1-\frac{\e}{2}}_{T_0+\frac{\e}{2}}\partial_t u^{\e}(x,t)\phi(x,t)dtdx+\Sigma(u,\e)\\
=\int_{\Omega}&\Big[u(x,T_0)\phi(x,T_0)-u^{\e}\big(x,T_0+\frac{\e}{2}\big)\phi\big(x,T_0+\frac{\e}{2}\big)\Big]dx\\
+&\int_{\Omega}\Big[u^{\e}\big(x,T_1-\frac{\e}{2}\big)\phi\big(x,T_1-\frac{\e}{2}\big)-u(x,T_1)\phi(x,T_1)\Big]dx\\
+&\int^{T_1}_{T_0}\langle f(\cdot,t),\phi^\e(\cdot,t)\rangle dt.
\end{split}
\end{equation}
By integration by parts, the term $\Sigma(u,\e)$ can be rewritten as
\[
\begin{split}
\Sigma(u,\e)=&-\int_{\Omega}\Big(\frac{1}{\e}\int^{T_0+\e}_{T_0}u(x,t)\zeta\big(\frac{T_0-t}{\e}+\frac{1}{2}\big)dt\Big)\phi\big(x,T_0+\frac{\e}{2}\big)dx\\
&+\int_{\Omega}\int^{T_0+\frac{\e}{2}}_{T_0-\frac{\e}{2}}\Big(\frac{1}{\e^2}\int^{\ell+\frac{\e}{2}}_{T_0}u(x,t)\zeta^\prime\big(\frac{\ell-t}{\e}\big)dt\Big)\phi(x,\ell)d\ell dx\\
&+\int_{\Omega}\Big(\frac{1}{\e}\int^{T_1}_{T_1-\e}u(x,t)\zeta\big(\frac{T_1-t}{\e}-\frac{1}{2}\big)dt\Big)\phi\big(x,T_1-\frac{\e}{2}\big)dx\\
&-\int_{\Omega}\int^{T_1+\frac{\e}{2}}_{T_1-\frac{\e}{2}}\Big(\frac{1}{\e^2}\int^{T_1}_{\ell-\frac{\e}{2}}u(x,t)\zeta^\prime\big(\frac{\ell-t}{\e}\big)dt\Big)\phi(x,\ell)d\ell dx,
\end{split}
\]
where we used the fact that $\zeta$ has compact support in $(-1/2,1/2)$. Using a change of variables, we may write
\[
\begin{split}
\Sigma(u,\e)=&-\int_{\Omega}\left(\int_{-\frac{1}{2}}^{\frac{1}{2}} u\left(x,T_0-\varepsilon\,\rho+\frac{\varepsilon}{2}\right)\,\zeta(\rho)\,d\rho\right)\,\phi\left(x,T_0+\frac{\varepsilon}{2}\right)\,dx\\
&+\int_{\Omega}\int^{\frac{1}{2}}_{-\frac{1}{2}}\Big(\int^{\rho}_{-\frac{1}{2}}u(x,\e\rho+T_0-\e\sigma)\zeta^\prime(\sigma)d\sigma\Big)\phi(x,\e\rho+T_0)d\rho dx\\
&+\int_{\Omega}\left(\int^{\frac{1}{2}}_{-\frac{1}{2}} u\left(x,T_1-\varepsilon\,\rho-\frac{\varepsilon}{2}\right)\,\zeta(\rho)\,d\rho\right)\,\phi\left(x,T_1-\frac{\varepsilon}{2}\right)\,dx\\
&-\int_{\Omega}\int^{\frac{1}{2}}_{-\frac{1}{2}}\Big(\int^{\frac{1}{2}}_{\rho}u(x,\e\rho+T_1-\e\sigma)\zeta^\prime(\sigma)d\sigma\Big)\phi(x,\e\rho+T_1)d\rho dx.
\end{split}
\]
For another solution $v$ we may write the exact same formula, i.e.,
\begin{equation}\label{test_v}
\begin{split}
\int^{T_1}_{T_0}\iint_{\R^n\times\R^n}&\Big(J_p\big(v(x,t)-v(y,t)\big)\Big)\Big(\phi^{\e}(x,t)-\phi^{\e}(y,t)\Big)d\mu dt\\
&+\int_{\Omega}\int^{T_1-\frac{\e}{2}}_{T_0+\frac{\e}{2}}\partial_t v^{\e}(x,t)\phi(x,t)dtdx+\Sigma(v,\e)\\
=\int_{\Omega}&\Big[v(x,T_0)\phi(x,T_0)-v^{\e}\big(x,T_0+\frac{\e}{2}\big)\phi\big(x,T_0+\frac{\e}{2}\big)\Big]dx\\
+&\int_{\Omega}\Big[v^{\e}\big(x,T_1-\frac{\e}{2}\big)\phi\big(x,T_1-\frac{\e}{2}\big)-v(x,T_1)\phi(x,T_1)\Big]dx\\
+&\int^{T_1}_{T_0}\langle f(\cdot,t),\phi^\e(\cdot,t)\rangle dt.
\end{split}
\end{equation}
We now subtract \eqref{test_v} from \eqref{test_u} and obtain
\begin{equation}\label{test_uv}
\begin{split}
\int^{T_1}_{T_0}\iint_{\R^n\times\R^n}&\Big(J_p\big(u(x,t)-u(y,t)\big)-J_p\big(v(x,t)-v(y,t)\big)\Big)\Big(\phi^{\e}(x,t)-\phi^{\e}(y,t)\Big)d\mu dt\\
&+\int_{\Omega}\int^{T_1-\frac{\e}{2}}_{T_0+\frac{\e}{2}}\partial_t \big(u^{\e}(x,t)-v^{\e}(x,t)\big)\phi(x,t)dtdx+\Sigma(u,\e)-\Sigma(v,\e)\\
=\int_{\Omega}&\Big[u(x,T_0)\phi(x,T_0)-u^{\e}\big(x,T_0+\frac{\e}{2}\big)\phi\big(x,T_0+\frac{\e}{2}\big)\Big]dx\\
-&\int_{\Omega}\Big[v(x,T_0)\phi(x,T_0)-v^{\e}\big(x,T_0+\frac{\e}{2}\big)\phi\big(x,T_0+\frac{\e}{2}\big)\Big]dx\\
+&\int_{\Omega}\Big[u^{\e}\big(x,T_1-\frac{\e}{2}\big)\phi\big(x,T_1-\frac{\e}{2}\big)-u(x,T_1)\phi(x,T_1)\Big]dx\\
-&\int_{\Omega}\Big[v^{\e}\big(x,T_1-\frac{\e}{2}\big)\phi\big(x,T_1-\frac{\e}{2}\big)-v(x,T_1)\phi(x,T_1)\Big]dx,
\end{split}
\end{equation}
for every $\phi\in L^{p}(J;X^{s,p}_0(\Omega,\Omega^\prime))\cap C^1(J;L^2(\Omega))$. With $F$ as the statement we now use \eqref{test_uv} with the choice
$$
\phi=F(u^{\e}-v^{\e})\eta:=F(u-v)^{\e}\eta.
$$
This yields
\begin{equation}\label{test_u-v}
\begin{split}
\int^{T_1}_{T_0}&\iint_{\R^n\times\R^n}\Big(J_p\big(u(x,t)-u(y,t)\big)-J_p\big(v(x,t)-v(y,t)\big)\Big)\\
&\times \Big( \big(F\big(u^{\e}(x,t)-v^{\e}(x,t)\big)\eta(t)\big)^{\e}-\big(F\big(u^{\e}(y,t)-v^{\e}(y,t)\big)\eta(t)\big)^\e\Big)d\mu dt\\
&+\int^{T_1-\frac{\e}{2}}_{T_0+\frac{\e}{2}}\int_{\Omega}\partial_t \big(u^{\e}(x,t)-v^{\e}(x,t)\big)F\big(u^{\e}(x,t)-v^{\e}(x,t)\big)\eta(t) dxdt+\Sigma(u,\e)-\Sigma(v,\e)\\
=\int_{\Omega}&\left[(u-v)\big(x,T_0\big)F\big(u-v\big)^{\e}(x,T_0)\eta(T_0)-(u-v)^{\e}\big(x,T_0+\frac{\e}{2}\big)F\big(u-v\big)^{\e}\big(x,T_0+\frac{\e}{2}\big)\eta\big(T_0+\frac{\e}{2}\big)\right]dx\\
+\int_{\Omega}&\left[(u-v)^{\e}\big(x,T_1-\frac{\e}{2}\big)F\big(u-v\big)^{\e}\big(x,T_1-\frac{\e}{2}\big)\eta\big(T_1-\frac{\e}{2}\big)-(u-v)(x,T_1)F\big(u-v\big)^{\e}(x,T_1)\eta(T_1)\right]dx.
\end{split}
\end{equation}

We observe that 
$$
\partial_t\big(u^{\e}(x,t)-v^{\e}(x,t)\big)F\big(u^{\e}(x,t)-v^{\e}(x,t)\big)\eta(t)=\partial_t\Big(\FF\big(u^{\e}(x,t)-v^{\e}(x,t)\big)\eta(t)\Big)-\FF\big(u^{\e}(x,t)-v^{\e}(x,t)\big)\eta'(t).
$$
Therefore,
\[
\begin{split}
&\int^{T_1-\frac{\e}{2}}_{T_0+\frac{\e}{2}}\int_{\Omega}\partial_t \big(u^{\e}(x,t)-v^{\e}(x,t)\big)F\big(u^{\e}(x,t)-v^{\e}(x,t)\big)\eta(t) dxdt\\
&=\int_{\Omega}\FF\big(u^{\e}(x,T_1-\frac{\e}{2})-v^{\e}(x,T_1-\frac{\e}{2})\big)\eta(T_1-\frac{\e}{2})dx\\
&-\int_{\Omega}\FF\big(u^{\e}(x,T_0+\frac{\e}{2})-v^{\e}(x,T_0+\frac{\e}{2})\big)\eta( T_0+\frac{\e}{2})dx-\int^{T_1-\frac{\e}{2}}_{T_0+\frac{\e}{2}}\int_{\Omega} \FF\big(u^{\e}(x,t)-v^{\e}(x,t)\big)\eta'(t)dx dt.
\end{split}
\]
By inserting this into \eqref{test_u-v} we obtain
\begin{equation}
\begin{split}\label{test_u_v_2}
\int^{T_1}_{T_0}&\iint_{\R^n\times\R^n}\Big(J_p\big(u(x,t)-u(y,t)\big)-J_p\big(v(x,t)-v(y,t)\big)\Big)\\
&\times  \Big( \big(F\big(u^{\e}(x,t)-v^{\e}(x,t)\big)\eta(t)\big)^{\e}-\big(F\big(u^{\e}(y,t)-v^{\e}(y,t)\big)\eta(t)\big)^\e\Big)d\mu dt\\
&+\int_{\Omega}\Big(\FF(u-v)^{\e}\big(x, T_1-\frac{\e}{2}\big)\eta(T_1-\frac{\e}{2})-\FF(u-v)^{\e}\big(x,T_0+\frac{\e}{2}\big)\eta(T_0+\frac{\e}{2})\Big)dx+\Sigma(u,\e)-\Sigma(v,\e)\\
=\int_{\Omega}&\left[(u-v)\big(x,T_0\big)F\big(u-v\big)^{\e}(x,T_0)\eta(T_0)-(u-v)^{\e}\big(x,T_0+\frac{\e}{2}\big)F\big(u-v\big)^{\e}\big(x,T_0+\frac{\e}{2}\big)\eta\big(T_0+\frac{\e}{2}\big)\right]dx\\
+\int_{\Omega}&\left[(u-v)^{\e}\big(x,T_1-\frac{\e}{2}\big)F\big(u-v\big)^{\e}\big(x,T_1-\frac{\e}{2}\big)\eta\big(T_1-\frac{\e}{2}\big)-(u-v)(x,T_1)F\big(u-v\big)^{\e}(x,T_1)\eta(T_1)\right]dx\\
+&\int^{T_1-\frac{\e}{2}}_{T_0+\frac{\e}{2}}\int_{\Omega} \FF\big(u^{\e}(x,t)-v^{\e}(x,t)\big)\eta'(t)dx dt.
\end{split}
\end{equation}

We now wish to pass $\e\to 0$ in \eqref{test_u_v_2}. We start with the second term in the right-hand side. We have
\[
\begin{split}
(u-v)^{\e}&\big(x,T_1-\frac{\e}{2}\big)F(u-v)^{\e}(x,T_1-\frac{\e}{2})-(u-v)(x,T_1)F(u-v)^{\e}(x,T_1)\\
&\leq C_{F}\Big|(u-v)^{\e}\big(x,T_1-\frac{\e}{2}\big)\Big|\Big|(u-v)^{\e}\big(x,T_1-\frac{\e}{2}\big) - (u-v)(x,T_1)\Big|\\
&+ C_{F}\big|(u-v)\big(x,T_1\big)\big|\Big|(u-v)^{\e}(x,T_1-\frac{\e}{2})-\big(u-v\big)^{\e}\big(x,T_1\big)\Big|,
\end{split}
\]
where $C_{F}$ denotes the Lipschitz constant of $F(t)$, which is indepent of $\e$. After integration this yields
\[
\begin{split}
\int_{\Omega}&(u-v)^{\e}\big(x,T_1-\frac{\e}{2}\big)F(u-v)^{\e}(x,T_1-\frac{\e}{2})-(u-v)(x,T_1)F(u-v)^{\e}(x,T_1))dx\\
&\leq C_{F}\int_{\Omega}\Big|(u-v)^{\e}\big(x,T_1-\frac{\e}{2}\big)\Big|\Big|(u-v)^{\e}\big(x,T_1-\frac{\e}{2}\big) - (u-v)(x,T_1)\Big|dx\\
&+C_{F}\int_{\Omega}\big|(u-v)\big(x,T_1\big)\big|\Big|(u-v)^{\e}(x,T_1-\frac{\e}{2})-\big(u-v\big)^{\e}\big(x,T_1\big)\Big|dx
\end{split}
\]
which converges to 0 as $\e\rightarrow 0$, by the fact that $u-v\in C_{\rm loc}((0,\infty);L^2(\Omega))$. Since $\eta(T_1-\frac{\e}{2})\to \eta(T_1)$ uniformly, this allows us to pass to the limit in this term. The same arguments may be applied to prove that the first term of the right-hand side also converges to zero.

For the term which is on the left-hand side of \eqref{test_u_v_2}
$$
\int_{\Omega}\FF (u-v)^{\e}\big(x,T_1-\frac{\e}{2}\big)\eta\big(T_1-\frac{\e}{2}\big)dx
$$
we estimate the following difference 
\[
\begin{split}
&\left|\int_{\Omega}\FF (u-v)^{\e}\big(x,T_1-\frac{\e}{2}\big)dx-\int_{\Omega}\FF (u-v)(x,T_1)dx\right|
\end{split}
\]
and follow the same steps as above to prove the convergence to 0 as $\e\rightarrow0$. The term containing $\FF (u-v)^{\e}\big(x,T_0+\frac{\e}{2}\big)$ and the last term in the right hand side may also be treated in a similar way. We omit the details.

Similar arguments also gives that the term $\Sigma(u,\e)-\Sigma(v,\e)\rightarrow 0$ as $\e\rightarrow 0$, using that
$$u-v\in L^p_{\rm loc}\big((0,\infty); W^{s,p}_0(\Omega)\big)\cap  C_{\rm loc}(I;L^2(\Omega)).
$$
In order to to treat the first term in the left-hand side we write
\[
\begin{split}
\mathcal{A}_\e=&\int^{T_1}_{T_0}\iint_{\R^n\times\R^n}\Big(J_p\big(u(x,t)-u(y,t)\big)-J_p\big(v(x,t)-v(y,t)\big)\Big)\\
&\times  \Big( \big(F\big(u^{\e}(x,t)-v^{\e}(x,t)\big)\eta(t)\big)^{\e}-\big(F\big(u^{\e}(y,t)-v^{\e}(y,t)\big)\eta(t)\big)^\e\Big)d\mu dt,
\end{split}
\]
and
\[
\begin{split}
\mathcal{A}=&\int^{T_1}_{T_0}\iint_{\R^n\times\R^n}\Big(J_p\big(u(x,t)-u(y,t)\big)-J_p\big(v(x,t)-v(y,t)\big)\Big)\\
&\times \Big( F\big(u(x,t)-v(x,t)\big)-F\big(u(y,t)-v(y,t)\big)\Big)\eta(t) d\mu dt.
\end{split}
\]
Then
\[
\begin{split}
|\mathcal{A}_{\e}-\mathcal{A}|=&\int^{T_1}_{T_0}\iint_{\R^n\times\R^n}\Big(J_p\big(u(x,t)-u(y,t)\big)-J_p\big(v(x,t)-v(y,t)\big)\Big)\\
&\times\Big(\big(F(u-v)^{\e}(x,t)\eta(t)\big)^{\e}-F(u-v)(x,t)\eta(t)\\
&-\big(\big(F(u-v)^{\e}(y,t)\eta(t)\big)^\e-F(u-v)(y,t)\eta(t)\big)\Big)d\mu dt.
\end{split}
\]
Recall that, being weak solutions, $u,v\in L^p_{\rm loc}(J;W^{s,p}(\Omega'))$ for some $\Omega\Subset\Omega^\prime$, and also $u-v$ vanishes outside $\Omega$. Therefore, 
\[
\begin{split}
|\mathcal{A}_{\e}-\mathcal{A}|&=\int^{T_1}_{T_0}\iint_{\Omega^\prime\times\Omega^\prime}\Big(J_p\big(u(x,t)-u(y,t)\big)-J_p\big(v(x,t)-v(y,t)\big)\Big)\\
&\times\Big( \big(F(u-v)^{\e}(x,t)\eta(t)\big)^{\e}-F(u-v)(x,t)\eta(t) \\
&-\big(\big(F(u-v)^{\e}(y,t)\eta(t)\big)^\e-F(u-v)(y,t)\eta(t)\big)\Big)d\mu dt\\
+2\int^{T_1}_{T_0}&\iint_{\Omega\times(\R^n\setminus\Omega^\prime)}\Big(J_p\big(u(x,t)-u(y,t)\big)-J_p\big(v(x,t)-v(y,t)\big)\Big)\\
&\times \Big(\big(F(u-v)^{\e}(x,t)\eta(t)\big)^{\e}-F(u-v)(x,t)\eta(t)\Big)d\mu dt:=\Theta_1(\e)+\Theta_2(\e).
\end{split}
\]
Since $F(t)$ is Lipschitz, $F(0)=0$ and $u-v\in L_{\rm loc}^p\big((0,\infty); W^{s,p}_0(\Omega')\big)$ we have 
$$
\int^{T_1}_{T_0}[\big(F(u-v)^{\e}(x,t)\eta(t)\big)^{\e}]^p_{W^{s,p}(\Omega')}dt\leq C_F.
$$
Hence, up to extracting a subsequence, 
$$
\big(F(u-v)^{\e}(x,t)\eta(t)\big)^{\e}\to F(u-v)(x,t)\eta(t),
$$
weakly in
$
L^p(J; W^{s,p}_0(\Omega')).
$
This means exactly that 
$$
\frac{\big(F(u-v)^{\e}(x,t)\eta(t)\big)^{\e}-\big(F(u-v)^{\e}(y,t)\eta(t)\big)^{\e}}{|x-y|^{\frac{N}{p}+s}}
$$
converges weakly in $L^{p}(J; L^{p}(\Omega'\times\Omega'))$. Using that 
$$
\frac{J_p\big(u(x,t)-u(y,t)\big)-J_p\big(v(x,t)-v(y,t)\big)}{|x-y|^{\frac{N}{p^\prime}+s(p-1)}}
$$
belongs to the space $L^{p^\prime}(J; L^{p^{\prime}}(\Omega'\times\Omega'))$, 
this implies
$$
\lim_{\e\rightarrow 0}\Theta_1(\e) = 0. 
$$
For $\Theta_2(\e)$, let $x\in \Omega$, $t\in [T_0,T_1]$ and define
$$
\mathcal{F}(x,t)=\int_{\R^n\setminus\Omega^\prime}\frac{J_p\big(u(x,t)-u(y,t)\big)-J_p\big(v(x,t)-v(y,t)}{|x-y|^{n+sp}} dy.
$$
Then 
\[
\begin{split}
|\mathcal{F}(x,t)|&\leq C(\Omega,\Omega')\int_{\R^n\setminus\Omega^\prime}\frac{|u(x,t)|^{p-1}+|v(x,t)|^{p-1}+|u(y,t)|^{p-1}+|v(y,t)|^{p-1}}{|x-y|^{n+sp}} dy \\
& \leq C(\Omega,\Omega')\left(|u(x,t)|^{p-1}+|v(x,t)|^{p-1}+\|u(\cdot,t)\|_{L^{p-1}_{sp}(\R^n)}^{p-1}+\|v(\cdot,t)\|_{L^{p-1}_{sp}(\R^n)}^{p-1}\right)\in L^{p'}(J;L^{p'}(\Omega)).
\end{split}
\]

Since $F(t)$ is Lipschitz continuous, $F(0)=0$ and $u,v\in L_{\rm loc}^p\big((0,\infty));L^p(\Omega)\big)$, we have
$$
\int^{T_1}_{T_0}\|\big(F(u-v)^{\e}(x,t)\eta(t)\big)^{\e}-F(u-v)(x,t)\eta(t)\|_{L^p(\Omega)} dt\leq C_F.
$$
We may therefore extract a subsequence such that $\big(F(u-v)^{\e}\eta\big)^{\e}\to F(u-v)\eta$ weakly in $L^p\big(J;L^p(\Omega)\big)$. This permits to conclude
$$
\lim_{\e\rightarrow 0}\Theta_2(\e) = 0.
$$
\end{proof}

\subsection{Comparison principle}
In this section, we present two comparison principles. The proofs of these two propositions are almost identical with the proof of Proposition A.6 in \cite{BLS21}. We provide some details below. We first introduce some notation from \cite{BLS21}.

As before we assume that $\Omega\Subset\Omega^\prime\subset\R^n$, where $\Omega^\prime$ is a bounded open set in $\R^n$. We define
$$
X^{s,p}_{\psi}(\Omega,\Omega^\prime)=\big\{v\in W^{s,p}(\Omega^\prime)\cap L^{p-1}_{sp}(\R^n): v=\psi\  on\ \R^n\setminus\Omega\big\},
$$
where 
$$
\psi\in W^{s,p}(\Omega^\prime)\cap L^{p-1}_{sp}(\R^n).
$$
If $\psi=0$, we have that $X^{s,p}_0(\Omega, \Omega^\prime)\subset W^{s,p}(\Omega^\prime)$.

Now we give the definition of a weak solution of the initial boundary value problem as in \cite{BLS21}. Let $I=[t_0,t_1]$.
Assume that
\[
\begin{split}
&u_0\in L^2(\Omega),\\
&f\in L^{p^\prime}(I;(X^{s,p}_0(\Omega,\Omega^\prime))^*),\\
&g\in L^p(I;W^{s,p}(\Omega^\prime))\cap L^{p-1}(I;L^{p-1}_{sp}(\R^n))\textup{ and }\partial_t g\in L^{p^\prime}(I; [W^{s,p}(\Omega^\prime)]^*).
\end{split}
\]
We say that $u$ is a weak solution of the initial boundary value problem
$$
\begin{cases}
u_t+(-\lap_p)^s u = f& \text{ in }\Omega\times I,\\
u=u_0&\text{ in } \Omega\times\{t_0\},\\
u = g &\text{ in } (\R^n\setminus\Omega)\times I,
\end{cases}
$$
if the following properties are satisfied:
\begin{itemize}
\item[1.] $u\in L^p(I;W^{s,p}(\Omega^\prime))\cap L^{p-1}(I; L^{p-1}_{sp}(\R^n))\cap C(I;L^2(\Omega))$;
\item[2.] $u\in X^{s,p}_{g}(\Omega,\Omega^\prime)$ for almost every $t\in I$;
\item[3.] $\lim_{t\rightarrow t_0}\|u(\cdot,t)-u_0\|_{L^2(\Omega)}=0$;
\item[4.] for every $J:=[T_0, T_1]\subset I$ and every $\phi\in L^{p}(J;X^{s,p}_0(\Omega,\Omega^\prime))\cap C^1(J;L^2(\Omega))$
\[
\begin{split}
-\int^{T_1}_{T_0}&\int_{\Omega}u(x,t)\partial_t\phi(x,t)dxdt\\
&+\int^{T_1}_{T_0}\iint_{\R^n\times\R^n}\big(J_p(u(x,t)-u(y,t))\times(\phi(x,t)-\phi(y,t))\big)d\mu dt\\
&=\int_{\Omega}u(x,T_0)\phi(x,T_0)dx-\int_{\Omega}u(x,T_1)\phi(x,T_1)dx + \int^{T_1}_{T_0}\langle f(\cdot,t),\phi(\cdot,t)\rangle dt.
\end{split}
\]
\end{itemize}

\begin{prop}[Comparison with subsolutions: degenerate case]\label{prop:comp1}
Let $p\geq2$, $I=[t_0,t_1]$, $\Omega\Subset\Omega^\prime$ and suppose that $v\in L^\infty(I;L^\infty(\R^n))$ is a weak subsolution of \eqref{locweaksol} in $\Omega\times I$ satisfying 
$$
v\in L^p(I; W^{s,p}(\Omega^\prime))\cap C(I;L^2(\Omega)),\quad \partial_t v\in L^{p^\prime}(I;(X^{s,p}_0(\Omega,\Omega^\prime))^*),
$$
$$
\lim\limits_{t\rightarrow t_0}\|v(\cdot,t)-v_0\|_{L^2(\Omega)}=0\quad for\ some\ v_0\in L^2(\Omega). 
$$
Consider the unique weak solution $u$ to the initial boundary value problem
\[
\begin{cases}
u_t+(-\lap_p)^s u = f& \text{ in }\Omega\times I,\\
u(\cdot,t_0)=v_0&\text{ in } \Omega,\\
u = v&\text{ in } (\R^n\setminus\Omega)\times I,
\end{cases}
\]
where
$$
f\in L^{p^\prime}(I; (W^{s,p}(\Omega^\prime))^*).
$$
Then 
$$
u\geq v, \qquad in\ \R^n\times I.
$$

\end{prop}

\begin{proof}
Take $J=[T_0,T_1]\Subset(t_0,t_1)$, then as in the first part of the Lemma \ref{lem:test}, we obtain
\[
\begin{split}
\int^{T_1}_{T_0}&\iint_{\R^n\times\R^n}\Big(J_p(v(x,t)-v(y,t))-J_p(u(x,t)-u(y,t))\Big)\Big(\phi^{\e}(x,t)-\phi^\e(y,t)\Big)d\mu dt\\
&+\int_{\Omega}\int^{T_1-\frac{\e}{2}}_{T_0+\frac{\e}{2}}\partial_t(v^{\e}(x,t)-u^\e(x,t))\phi(x,t)dtdx+\Sigma(\e)\\
&\leq\int_{\Omega}\Big[(v(x,T_0)-u(x,T_0))\phi(x,T_0)-(v^\e(x,T_0+\frac{\e}{2})-u^\e(x,T_0+\frac{\e}{2}))\phi(x,T_0+\frac{\e}{2})\Big]dx\\
&\quad+\int_{\Omega}\Big[(v^\e(x,T_1-\frac{\e}{2})-u^\e(x,T_1-\frac{\e}{2}))\phi(x,T_1-\frac{\e}{2})-(v(x,T_1)-u(x,T_1))\phi(x,T_1)\Big]dx,
\end{split}
\]
for every nonnegative $\phi\in L^p((T_0,T_1); X^{s,p}_0(\Omega, \Omega^\prime))\cap C^1((T_0,T_1);L^2(\Omega))$. The quantity $\Sigma(\e)$ is defined as in the proof of Lemma \ref{lem:test}, with $v-u$ in place of $u$. Note that since $v$ is a subsolution, we obtain an inequality above, instead of an equality. Take the test function $\phi$ by\footnote{This choice of $\phi$ is merely Lipschitz in time. By density argument it can be seen that Lipschitz functions are also admissible as test functions.}
$$
\phi(x,t) = \left(v^\e(x,t)-u^\e(x,t)\right)_+,
$$
which gives
\[
\begin{split}
\int_{\Omega}&\int^{T_1-\frac{\e}{2}}_{T_0+\frac{\e}{2}}\partial_t(v^\e(x,t)-u^\e(x,t))\phi(x,t)dtdx\\
&=\int_{\Omega}\int^{T_1-\frac{\e}{2}}_{T_0+\frac{\e}{2}}\partial_t(v^\e(x,t)-u^\e(x,t))(v^\e(x,t)-u^\e(x,t))_+dtdx\\
&=\int_{\Omega}\int^{T_1-\frac{\e}{2}}_{T_0+\frac{\e}{2}}\frac{1}{2}\partial_t\big((v^\e(x,t)-u^\e(x,t))_+\big)^2dtdx\\
&=\frac{1}{2}\Big[\int_\Omega\big(v^\e(x,T_1-\frac{\e}{2})-u^\e(x,T_1-\frac{\e}{2})\big)^2_+dx
-\int_\Omega\big(v^\e(x,T_0+\frac{\e}{2})-u^\e(x,T_0+\frac{\e}{2})\big)^2_+dx\Big].\\
\end{split}
\]
As in the proof of Lemma \ref{lem:test}, we have that 
\[
\begin{split}
&\int_\Omega\big[\big(v(x,T_0)-u(x,T_0)\big)\phi(x,T_0)-\big(v^\e(x,T_0+\frac{\e}{2})-u^\e(x,T_0+\frac{\e}{2})\big)\phi(x,T_0+\frac{\e}{2})\big]dx\\
+&\int_\Omega\big[\big(v^\e(x,T_1-\frac{\e}{2})-u^\e(x,T_1-\frac{\e}{2})\big)\phi(x,T_1-\frac{\e}{2})-\big(v(x,T_1)-u(x,T_1)\big)\phi(x,T_1)\big]dx,
\end{split}
\]
tends to $0$ $\e\rightarrow 0$. Passing $\e\rightarrow 0$, we obtain 
\[
\begin{split}
\int^{T_1}_{T_0}&\iint_{\R^n\times\R^n}\big(J_p(v(x,t)-v(y,t))-J_p(u(x,t)-u(y,t))\big)\\
&\times\big((v(x,t)-u(x,t))_+-(v(y,t)-u(y,t))_+\big)d\mu dt\\
&+\frac{1}{2}\left[\int_{\Omega}(v(x,T_1)-u(x,T_1))^2_+dx-\int_{\Omega}(v(x,T_0)-u(x,T_0))^2_+dx\right]\leq 0.
\end{split}
\]
By the Lemma A.3 in \cite{BLS}
\[
\begin{split}
\big(&J_p(a-b)-J_p(c-d)\big)\big((a-c)_+-(b-d)_+\big)\\
&\geq C_p|(a-b)-(c-d)|^{p-1}|(a-c)_+-(b-d)_+|\geq C^\prime_p|(a-c)_+-(b-d)_+|^p,
\end{split}
\]
with constants $C_p>0$ and $C^\prime_p>0$. Therefore, for all $T_0$ and $T_1$ such that  $t_0<T_0<T_1<t_1$, there holds
\[
\begin{split}
C&\int^{T_1}_{T_0}\big[(v-u)_+\big]^p_{W^{s,p}(\R^n)}dt\\
&\leq\frac{1}{2}\left[\int_{\Omega}(v(x,T_0)-u(x,T_0))^2_+dx-\int_{\Omega}(v(x,T_1)-u(x,T_1))^2_+dx\right].
\end{split}
\]
Using the hypothesis on $v$ and $u$ together with the monotone convergence theorem, we may pass $T_0\rightarrow t_0$ and obtain
$$
0\leq C\int^{T_1}_{t_0}\big[(v-u)_+\big]^p_{W^{s,p}(\R^n)}dt\leq -\frac{1}{2}\int_{\Omega}\big(v(x,T_1)-u(x,T_1)\big)^2_+dx.
$$
This implies that
$$
u(x,T_1)\geq v(x,T_1)
$$
 for a.e. $x\in\Omega$. Since $T_1$ is arbitrary, the proof is complete.

\end{proof}

\begin{prop}[Comparison with subsolutions: singular case]\label{prop:comp2}
Let $1<p<2$, $I=[t_0,t_1]$, $\Omega\Subset\Omega^\prime$ and suppose that $v\in L^\infty(I;L^\infty(\R^n))$ is a weak subsolution of \eqref{locweaksol} in $\Omega\times I$ satisfying 
$$
v\in L^p(I; W^{s,p}(\Omega^\prime))\cap C(I;L^2(\Omega)),\quad \partial_t v\in L^{p^\prime}(I;(X^{s,p}_0(\Omega,\Omega^\prime))^*),
$$
$$
\lim\limits_{t\rightarrow t_0}\|v(\cdot,t)-v_0\|_{L^2(\Omega)}=0\quad for\ some\ v_0\in L^2(\Omega). 
$$
Consider the unique weak solution $u$ to the initial boundary value problem
\[
\begin{cases}
u_t+(-\lap_p)^s u = f& \text{ in }\Omega\times I,\\
u(\cdot,t_0)=v_0&\text{ in } \Omega,\\
u = v&\text{ in } (\R^n\setminus\Omega)\times I,
\end{cases}
\]
where
$$
f\in L^{p^\prime}(I; (W^{s,p}(\Omega^\prime))^*).
$$
Then 
$$
u\geq v, \qquad in\ \R^n\times I.
$$

\end{prop}

\begin{proof}
By following the same steps as in the proof of Proposition \ref{Prop:Degenerate}, we obtain
\[
\begin{split}
\int^{T_1}_{T_0}&\iint_{\R^n\times\R^n}\big(J_p(v(x,t)-v(y,t))-J_p(u(x,t)-u(y,t))\big)\\
&\times\big((v(x,t)-u(x,t))_+-(v(y,t)-u(y,t))_+\big)d\mu dt\\
&\leq-\frac{1}{2}\left[\int_{\Omega}(v(x,T_1)-u(x,T_1))^2_+dx-\int_{\Omega}(v(x,T_0)-u(x,T_0))^2_+dx\right].
\end{split}
\]
Since 
$$
|a|^{p-2}a-|b|^{p-2}b=(p-1)\int^{a}_{b}|t|^{p-2}dt=c(|a|,|b|,p)(a-b),
$$
in which $c(|a|,|b|,p)$ is a non-negtive term depending on $|a|$, $|b|$ and $p$, for the nonlocal term in the first integral on the left hand side we have that
\[
\begin{split}
\int^{T_1}_{T_0}&\iint_{\R^n\times\R^n}\big(J_p(v(x,t)-v(y,t))-J_p(u(x,t)-u(y,t))\big)\\
&\times\big((v(x,t)-u(x,t))_+-(v(y,t)-u(y,t))_+\big)d\mu dt\\
&=(p-1)\int^{T_1}_{T_0}\iint_{\R^n\times\R^n}c(|v|,|u|,p)\big(v(x,t)-v(y,t)-(u(x,t)-u(y,t))\big)\\
&\times\big((v(x,t)-u(x,t))_+-(v(y,t)-u(y,t))_+\big)d\mu dt.
\end{split}
\]
Then since
\[
\begin{split}
(a-b)&(a_+-b_+)=\big(a_++a_--(b_++b_-)\big)(a_+-b_+)\\
&=a^2_++b^2_+-2a_+b_+-a_+b_--a_-b_+\geq(a_+-b_+)^2
\end{split}
\]
where $t_+$ and $t_-$ denote the positive part and negative part of $t$ and $t=t_++t_-$, we have that
\[
\begin{split}
0\leq(p-1)\int^{T_1}_{T_0}&\iint_{\R^n\times\R^n}c(|v|,|u|,p)\big((v(x,t)-u(x,t))_+-(v(y,t)-u(y,t))_+\big)^2d\mu dt\\
&\leq\int^{T_1}_{T_0}\iint_{\R^n\times\R^n}\big(J_p(v(x,t)-v(y,t))-J_p(u(x,t)-u(y,t))\big)\\
&\times\big((v(x,t)-u(x,t))_+-(v(y,t)-u(y,t))_+\big)d\mu dt\\
&\leq-\frac{1}{2}\big[\int_{\Omega}(v(x,T_1)-u(x,T_1))^2_+dx-\int_{\Omega}(v(x,T_0)-u(x,T_0))^2_+dx\big].
\end{split}
\]
By letting $T_0\to t_0$, we obtain the desired result.
\end{proof}

\section{Pointwise inequalities}

\begin{lem}\label{lem:ineq}
Let $p>2$ and $q\ge 1$. For every $a,b\in\mathbb{R}$ we have
\begin{equation}
J_{q}(a-b)\,\Big(J_p(a)-J_p(b)\Big)\ge (p-1)\,\left(\frac{q}{p-2+q}\right)^q\, \left||a|^\frac{p-2}{q} a-|b|^\frac{p-2}{q} \right|^q.
\end{equation}
\end{lem}
This is Lemma A.1 in \cite{BLS}.

\begin{lem}\label{lem:ineq2}
Let $p\ge 2$, $\gamma\ge 1$ and $a,b,c,d\in\mathbb{R}$. Then
\begin{equation}
\label{erik}
\begin{split}
\Big(J_p(a-c)-J_p(b-d)\Big)&\Big(J_{\gamma+1}(a-b)-J_{\gamma+1}(c-d)\Big)\\
&\ge C\, \Big||a-b|^\frac{\gamma-1}{p}\,(a-b)-|c-d|^\frac{\gamma-1}{p}\,(c-d)\Big|^p,
\end{split}
\end{equation}
where
$$
C(p,\gamma)=\frac{\gamma}{3\cdot 2^{p-1}}\left(\frac{p}{\gamma-1+p}\right)^p.
$$ 
\end{lem}
This is Lemma A.5 in \cite{BLS}. If we track the constants carefully we see that 
$$
C(p,\gamma)=\frac{\gamma}{3\cdot 2^{p-1}}\left(\frac{p}{\gamma-1+p}\right)^p.
$$ \\
Note that there is an exponent 2 which should be $q$ in the proof of this in \cite{BLS}. Here we have only used the estimated
$$
|J_p(t)-J_p(s)|\geq \frac{1}{3\cdot 2^{p-1}}|t-s|^{p-1}, 
$$
which follows from Lemma 2 in \cite{Lin14}.\\

\begin{lem}\label{lem:ineqM}
Let $p\ge 2$, $\gamma\ge 1$ and $a,b,c,d\in\mathbb{R}$. Then
\begin{equation}
\label{erikM}
\begin{split}
\Big(J_p(a-c)-J_p(b-d)\Big)&\Big(J^M_{\gamma+1}(a-b)-J^M_{\gamma+1}(c-d)\Big)\\
&\ge C\, \Big||a-b|^\frac{\gamma-1}{p}\,(a-b)_M-|c-d|^\frac{\gamma-1}{p}\,(c-d)_M\Big|^p,
\end{split}
\end{equation}
where
$$
C(p,\gamma)=\frac{\gamma}{3\cdot 2^{p-1}}\left(\frac{p}{\gamma-1+p}\right)^p.
$$ 
\end{lem}
\begin{proof} By symmetry, it is enough to treat the following cases:\\
\noindent {\bf Case 1: } $|c-d|\leq M, a-b>M$. Then the LHS becomes
$$
\Big(J_p(a-c)-J_p(b-d)\Big)\Big(J_{\gamma+1}(M)-J_{\gamma+1}(c-d)\Big).
$$
Note that in this case, both factors are non-negative. Also, since $a>M+b$ we can bound this from below by
$$
\Big(J_p(M+b-c)-J_p(b-d)\Big)\Big(J_{\gamma+1}(M)-J_{\gamma+1}(c-d)\Big)\geq  C\, \Big||M|^\frac{\gamma-1}{p}\,M-|c-d|^\frac{\gamma-1}{p}\,(c-d)_M\Big|^p,
$$
by Lemma \ref{lem:ineq2}. This proves the inequality in this case.

\noindent {\bf Case 2: } $|c-d|\leq M, a-b<-M$. Then the LHS becomes
$$
\Big(J_p(a-c)-J_p(b-d)\Big)\Big(J_{\gamma+1}(-M)-J_{\gamma+1}(c-d)\Big).
$$
Note that in this case, both factors are non-positive. Since $a>-M+b$ we can bound this from below by
$$
\Big(J_p(-M+b-c)-J_p(b-d)\Big)\Big(J_{\gamma+1}(-M)-J_{\gamma+1}(c-d)\Big)\geq  C\, \Big||M|^\frac{\gamma-1}{p}\,(-M)-|c-d|^\frac{\gamma-1}{p}\,(c-d)_M\Big|^p,
$$
by Lemma \ref{lem:ineq2}. This proves the inequality in this case.

\noindent {\bf Case 3: } $c-d<- M, a-b> M$. Then the LHS becomes
$$
\Big(J_p(a-c)+J_p(d-b)\Big)\Big(J_{\gamma+1}(M)+J_{\gamma+1}(M)\Big).
$$
Using that $a>b+M$ and $d>c+M$ we obtain the lower bound
$$
\Big(J_p(M+b-c)+J_p(M+c-b)\Big)\Big(J_{\gamma+1}(M)+J_{\gamma+1}(M)\Big) \geq  C\, \Big||M|^\frac{\gamma-1}{p}\,M+|M|^\frac{\gamma-1}{p}\,M\Big|^p,
$$
by Lemma \ref{lem:ineq2}. This proves the inequality in this case.
\end{proof}

\begin{lem} \label{lemma:simple inequality}
Let $p>1$ and $a,b \in \R^n$. Then
\[
\left||a|^{p-2}a-|b|^{p-2}b\right| \leq c\,(|b|+|a-b|)^{p-2}|a-b|,
\]
where $c$ depends only on $p$.
\end{lem}
This is Lemma 3.4 in \cite{KKL}.

\begin{lem} \label{lemma:sing_ineq_1}
Let $2>p>1$ and $a,b \in \R^n$. Then
\[
(p-1)\frac{|a-b|^2}{(|a|+|b|)^{2-p}} \leq \left(|a|^{p-2}a-|b|^{p-2}b\right)\cdot (a-b).
\]
\end{lem}
This is inequality $(2.3)$ in \cite{JL09}.

\begin{lem} \label{lemma:sing_ineq_2}
Let $\gamma\geq 2$, $2>p>1$ and $a,b,c,d \in \R$. Then
\[
\begin{split}
    \big(J_\gamma(a-b)&-J_\gamma(c-d)\big)\big(J_p(a-c)-J_p(b-d)\big) \\
    &\geq 4\frac{(\gamma-1)(p-1)}{\gamma^2}\left||a-b|^{\frac{\gamma-2}{2}}(a-b)-|c-d|^{\frac{\gamma-2}{2}}(c-d)\right|^2\big(|a-c|+|b-d|\big)^{p-2}.
\end{split}
\]
\end{lem}
\begin{proof}
The proof is just a combination of Lemma \ref{lemma:sing_ineq_1} and Lemma \ref{lem:ineq}.
\end{proof}

\begin{lem} \label{lemma:Msing_ineq_2}
Let $\gamma \geq 2$, $2>p>1$ and $a,b,c,d \in \R$. Then
\[
\begin{split}
    \big(J^M_\gamma(a-b)&-J^M_\gamma(c-d)\big)\big(J_p(a-c)-J_p(b-d)\big) \\
    &\geq 4\frac{(\gamma-1)(p-1)}{\gamma^2}\left||a-b|^{\frac{\gamma-2}{2}}(a-b)_M-|c-d|^{\frac{\gamma-2}{2}}(c-d)_M\right|^2\big(|a-c|+|b-d|\big)^{p-2}.
\end{split}
\]
\end{lem}
\begin{proof} By symmetry, it is enough to treat the following cases:\\
\noindent {\bf Case 1: } $|c-d|\leq M, a-b>M$. Then the LHS becomes
$$
\Big(J_p(a-c)-J_p(b-d)\Big)\Big(J_{\gamma}(M)-J_{\gamma}(c-d)\Big).
$$
Note that in this case, both factors are non-negative. Also, since $a>M+b$ we can bound this from below by
\[
\begin{split}
\Big(J_p(M+b-c)&-J_p(b-d)\Big)\Big(J_{\gamma}(M)-J_{\gamma}(c-d)\Big)\\ \geq  &4\frac{(\gamma-1)(p-1)}{\gamma^2}\left||M|^{\frac{\gamma-2}{2}}M-|c-d|^{\frac{\gamma-2}{2}}(c-d)_M\right|^2\big(|a-c|+|b-d|\big)^{p-2},
\end{split}
\]
by Lemma \ref{lemma:sing_ineq_2}. This proves the inequality in this case.

\noindent {\bf Case 2: } $|c-d|\leq M, a-b<-M$. Then the LHS becomes
$$
\Big(J_p(a-c)-J_p(b-d)\Big)\Big(J_{\gamma}(-M)-J_{\gamma}(c-d)\Big).
$$
Note that in this case, both factors are non-positive. Since $a>-M+b$ we can bound this from below by
\[
\begin{split}
\Big(J_p(-M+b-c)&-J_p(b-d)\Big)\Big(J_{\gamma}(-M)-J_{\gamma}(c-d)\Big)\\ &\geq  4\frac{(\gamma-1)(p-1)}{\gamma^2}\left||M|^{\frac{\gamma-2}{2}}(-M)-|c-d|^{\frac{\gamma-2}{2}}(c-d)_M\right|^2\big(|a-c|+|b-d|\big)^{p-2},
\end{split}
\]
by Lemma \ref{lemma:sing_ineq_2}. This proves the inequality in this case.

\noindent {\bf Case 3: } $c-d<- M, a-b> M$. Then the LHS becomes
$$
\Big(J_p(a-c)+J_p(d-b)\Big)\Big(J_{\gamma}(M)+J_{\gamma}(M)\Big).
$$
Using that $a>b+M$ and $d>c+M$ we obtain the lower bound
\[
\begin{split}
\Big(J_p(M+b-c)&+J_p(M+c-b)\Big)\Big(J_{\gamma}(M)+J_{\gamma}(M)\Big)\\ &\geq 4\frac{(\gamma-1)(p-1)}{\gamma^2}\left||M|^{\frac{\gamma-2}{2}}M+|c-d|^{\frac{\gamma-2}{2}}M\right|^2\big(|a-c|+|b-d|\big)^{p-2}
\end{split}
\]
by Lemma \ref{lemma:sing_ineq_2}. This proves the inequality in this case.
\end{proof}

\bibliographystyle{amsrefs}
\bibliography{refnonlocal}

 \bigskip
\paragraph{Acknowledgments:} Erik Lindgren was supported by the Swedish Research Council, 2017-03736. We thank the referee for carefully reviewing this work and for coming with several good points on how to improve our manuscript.
\bigskip

\noindent {\textsf{Feng Li\\  Department of Mathematics\\ Uppsala University\\ Box 480\\
751 06 Uppsala, Sweden}  \\
\textsf{e-mail}: feng.li@math.uu.se\\

\noindent {\textsf{Erik Lindgren\\  Department of Mathematics\\ Uppsala University\\ Box 480\\
751 06 Uppsala, Sweden}  \\
\textsf{e-mail}: erik.lindgren@math.uu.se\\

\end{document}